\documentclass[12pt]{article}
\usepackage{graphicx}
\usepackage{subfigure}
\usepackage{amssymb,amsmath,amsthm,epsfig}
\usepackage{latexsym, enumerate}
\usepackage{eepic}
\usepackage{epic}
\usepackage{graphicx}
\usepackage{color}

\usepackage{ifpdf}
\usepackage{booktabs}

\usepackage{dsfont}
\usepackage{multirow}
\usepackage{bm}
\usepackage{caption}

\theoremstyle{plain}
\newtheorem{lem}{Lemma}[section]
\newtheorem{thm}[lem]{Theorem}

\theoremstyle{definition}

\theoremstyle{remark}
\newtheorem{rem}{Remark}[section]

\begin{document}
\title{\large\bf Identification of the reaction coefficient in time fractional diffusion equations}
\author{
Xiaoyan Song\thanks
{College of Mathematics and Econometrics, Hunan University, Changsha 410082, China.
Email: xiaoyansong@hnu.edu.cn}
\and
Guang-Hui Zheng\thanks
{College of Mathematics and Econometrics, Hunan University, Changsha 410082, China.
Email: zhgh1980@163.com}
\and
Lijian Jiang\thanks
{School  of Mathematical Sciences, Tongji University, Shanghai 200092, China. Email: ljjiang@tongji.edu.cn. Corresponding author}
}
\date{}
\maketitle

\begin{center}{\bf ABSTRACT}
\end{center}\smallskip
In this paper, we  present an inverse problem of identifying  the reaction coefficient for time fractional diffusion equations in two dimensional spaces by using boundary Neumann data.
It is  proved that the forward operator is continuous with respect to the unknown parameter. Because the inverse problem is often ill-posed,   regularization strategies are imposed on the least fit-to-data functional to overcome the stability issue.
There may exist various kinds of functions to
reconstruct. It is crucial to choose a suitable regularization method.
We present a multi-parameter
 regularization  $L^{2}+BV$ method for the inverse problem. This can extend
the applicability for reconstructing the unknown functions.
Rigorous analysis is carried out for the inverse problem.
In particular, we analyze the existence and stability
of regularized  variational problem and the convergence.  To reduce the dimension in the inversion for numerical simulation, the unknown coefficient is represented by a suitable set of  basis functions based on a priori information.
 A few numerical examples  are presented for the inverse problem in time fractional diffusion equations to confirm the theoretic analysis and
the efficacy  of the different regularization methods.

\smallskip
{\bf keywords}: time fractional diffusion equation,  reaction inversion,  multi-parameter regularization

\section{Introduction}
Let  $\Omega$  be an open bounded domain in $\mathbb{R}^2$ with a Lipschitz boundary $\partial \Omega$ and $\nu$ be the outward unit normal vector to $\partial \Omega$. Define $\frac{\partial u}{\partial \nu}$=$\nabla u \cdot \nu$. Let $T>0$ be a fixed time length.
Then we consider the time fractional diffusion equation(TFDE) with a reaction term as follows
\begin{eqnarray}
\label{model-fpde}
\begin{cases}
{}_0 D_t^\alpha u(x,t)-\Delta u(x,t)+q(x)u(x,t)=0  \ \ \text{in} \ \  \Omega \times (0,T], \\
u(x,0)=0  \ \ \text{in} \ \ \Omega, \\
u(x,t)=\lambda(t)g(x) \ \ \text{on} \ \  \partial \Omega \times (0,T].
\end{cases}
\end{eqnarray}
where $\alpha \in (0,1)$ is the  fractional order of the derivative in time. Here ${}_0 D_t^\alpha u $ refers to the Caputo derivative \cite{caputo-deri-2,caputo-deri-1} with respect to $t$, i.e.,
\begin{equation}\label{caputa-def}
{}_0 D_t^\alpha u =\frac {1}{\Gamma (1-\alpha)}\int_0^t \frac {\partial u(x,s)}{\partial s}\frac {ds}{(t-s)^\alpha},
\end{equation}
where $\Gamma$ is the Gamma function.

The TFDEs generalize  standard diffusion equations through replacing the integer-order time derivative with a fractional derivative. Compared to the classical derivatives, fractional derivatives are used to simulate anomalous diffusion, where particles spread in a power-law manner \cite{model-applica-1}. The mean square displacement of particles from the original starting site is non-linear growth in time but verifies a generalized Fick's second law. Subdiffusion motion is characterized by an asymptotic long time behavior of the mean square displacement of the power-law pattern
\[
\langle x^{2}(t)\rangle\sim\frac{2K_{\gamma}}{\Gamma(1+\gamma)}t^{\gamma},\ \ t\rightarrow\infty,
\]
where $\gamma$ ($0<\gamma<1$) is the anomalous diffusion exponent and $K_{\gamma}$ is the generalized diffusion coefficient.
The TFDEs are widely used in materials, control, and system identification \cite{caputo-deri-1}.
Eq.(\ref{model-fpde}) can be used to model the anomalous diffusion phenomena in heterogeneous media, see \cite{anomalous-subdiffusion, model-applica-2, model-applica-1}.
This model describes a forward problem if $\lambda(t),g(x),q(x)$ are given.
There are many analysis methods and numerical methods to solve the forward problem, such as finite difference methods \cite{caputo-deri-2, forward-difference-fem-1, forward-difference-2} and finite element method \cite{forward-fem-1, forward-fem-2,forward-difference-fem-1}.
In this work, the mixed finite element method \cite{forward-mixedfem-2,mixed-fem-2,mixed-fem-1} is employed for numerically solving Eq.(\ref{model-fpde}).

In the paper,  we focus on  solving  the following inverse problem for the TFDE,

\emph{Inverse problem(IP)}: Identify the reaction coefficient $q(x)$ from all possible Cauchy data (Dirichlet-to-Neumann map $\Lambda(q)$ defined in Section 2) on boundary $\Gamma\subset\partial\Omega$.

The inverse problems of time fractional diffusion equations have attracted much attention in recent years.
It is  obvious that there are various unknown parameters to recover, such as coefficient, fractional order, the source term and so on.
For example, Cheng in \cite{fde-history-1} studied   a one-dimensional fractional diffusion equation with homogeneous Neumann boundary condition and proved that the uniqueness result for determining the fractional order and diffusion coefficient.
In \cite{fde-history-3}, Sakamoto \emph{et}  \emph{al.} considered  an initial value/boundary value problems for fractional diffusion-wave equation and built the uniqueness in recovering the initial value.
Tuan in \cite{fde-history-4} proved that it is  necessary to take a suitable initial distribution with only finite measurements on the boundary for uniquely recovering  the diffusion coefficient of a one-dimensional fractional diffusion equation.
Moreover, Jin \emph{et} \emph{al.}  in \cite{fde-history-5} considered an inverse problem of reconstructing a spatially dependent  potential term in a one-dimension time fractional diffusion equation through the flux measurements, which is similar to Eq.(\ref{model-fpde}).
Li \emph{et}  \emph{al.} in \cite{fde-history-6} investigated an inverse problem of identifying the space-dependent diffusion coefficient and the fractional order in the one-dimension time fractional diffusion equation.
Recently, Li \emph{et}  \emph{al.}  in \cite{paper-unique} considered an inverse problem for diffusion equations with multiple fractional time derivatives and proved the uniqueness in recovering the number of fractional time-derivative terms, the orders of the derivatives and spatially varying coefficients.
We can see \cite{fde-history-9, Y.Luchko-W.Rundell-2013, fde-history-7, fde-history-8,fde-history-2,G.H.zheng-2010,G.H.zheng-2014} for more literatures in the subject.
In this paper, we concentrate on  the reaction coefficient inverse problem in time fractional diffusion equation.
We note that there are not many  works on the reaction coefficient inverse problems for fractional diffusion equations in two dimension spatial spaces.

The solution to Eq.(\ref{model-fpde}) can be denoted by $u(x,t;q)$ in order to explicitly represent  its dependence on the unknown reaction coefficient $q$.
In this work, we use  Laplace transform with respect to $t$ to prove that the forward operator $F$ is continuous with respect to $q$.
It is  well known that inverse problems are usually ill-posed. To treat the ill-conditioned systems, a popular strategy  is to use a least-squares method by adding some penalty functions to the fit-to-data term,
see \cite{A.Doicu-2010, CR.Vogel-2002}.
However, because  there are various types of unknown parameters to recover, it is  a challenging task to choose a suitable regularization scheme based on a priori information.
The $L^{2}$ penalty \cite{A.Doicu-2010, CR.Vogel-2002}
is the most widely used in inverse problem, which can obtain a good regularization solution for smooth functions.
Total variation (TV) \cite{T.Chan-2003, A.Doicu-2010, sup-L^{1},CR.Vogel-2002} regularization method can penalize highly oscillatory solutions while allowing jumps in the regularization solution.
Moreover, there exist many other parameter regularization methods \cite{M.Belge-2002, K.Ito-2011}, which can improve the inversion reconstruction for difference situations.
In this paper, we  apply $L^{2}$, $BV$ and $L^{2}+BV$ regularization methods \cite{F.Ibarrola-2017} to reconstruct different kinds of the reaction coefficient $q$,  which can be
generalized  in a unified  variational representation.

In the work, we attempt to recover the unknown reaction coefficient for time fractional diffusion equations, and explore the related ill-posedness, regularization method and convergence property. We first prove that the forward operator is continuous with respect to the unknown parameter. Then  we introduce a multi-parameter regularization functional to recover different kinds of  unknown functions such as  smooth functions, functions with jumps and piecewise smooth functions.
Because Neumann data on the boundary is the measurement data for the inverse problem, we need to compute the flux on the boundary when solving the derived variational problem.
  To  get the robust   flux, we use mixed finite element method to solve the time fractional diffusion equations.

The paper is organized  as follows. In Section~2, we give some notations and preliminaries
for the paper.  A uniqueness result of the inverse problem is also presented in the section. In Section~3, we
show  that the continuity of the forward operator $F$. In Section~4, the regularization method is presented  and convergence results are provided. The mixed finite element method to solve the forward problem is presented in Section~5. In Section~6, a few numerical results are presented to confirm the presented analysis and algorithm. Some conclusions and comments are made finally.

\section{Notations and preliminaries}
In this section, we give some notations and preliminaries for the paper.
Let the admissible set as
\begin{equation}\label{admissible-set}
Q_{AD}:=\{q \in L^{\infty}(\Omega), 1/k_{2}\leq q(x)\leq k_{1}\},
\end{equation}
where $k_{1}$  and $k_{2}$ are some positive constants.
In the paper, we use the usual notations for Sobolev spaces \cite{sobolev space}.

We define the forward operator and the Dirichlet-to-Neumann(D-N) map, respectively, by
\begin{equation}\label{D-N map}
\begin{aligned}
&F(\cdot,g):Q_{AD}\mapsto L^{2}([0,T];H^{2}(\Omega)),\\
&\Lambda(q)g:=\frac{\partial u}{\partial \nu}|_{\Gamma \times (0,T]}\in L^{2}([0,T];H^{\frac {1}{2}}(\Gamma)),
\ \ g\in H^{\frac{3}{2}}(\partial\Omega).
\end{aligned}
\end{equation}
From the definitions above, the relationship of $F$ and D-N map can be represented as
\[
\Lambda(q)g=\Upsilon_{\Gamma}^{N}F(q,g),
\] where $\Upsilon_{\Gamma}^{N}$ is the Neumann trace operator.
For $\theta\in (0,\frac{\pi}{2})$ and $T>0$, we set
\[
\Omega_{\theta}:=\{z\in \mathds{C};z\neq0,\mid \arg z\mid<\theta\}.
\]

Next, we state a  result about  the uniqueness for the \emph{IP}.
\begin{lem} \cite{paper-unique}
\label{lem-uniqueness}
Let $\Gamma \subset \partial \Omega$ be an arbitrarily given subboundary and let $\gamma>2$ be arbitrarily fixed. Assume that for some $\theta\in (0,\frac{\pi}{2})$ the function $\lambda\not\equiv0$ can be analytically extended to $\Omega_{\theta}$ with $\lambda(0)=0$ and $\lambda^{'}(0)=0$ and there exists a constant $C>0$ such that $|\lambda^{(k)}(t)| \leq Ce^{Ct},\ t>0,\ 0\leq k\leq2$. We set
\[
\mathcal{U}=\{q\in W^{1,\gamma}(\Omega),\gamma>2, q\geq0 \ \ on \ \ \overline{\Omega}\}.
\]
Then if $q_{1}, \ q_{2} \in \mathcal{U}$ and
\[
\Lambda(q_{1})g=\Lambda(q_{2})g \ \ on \ \ \Gamma,
\]
for all $g\in H^{\frac{3}{2}}(\partial \Omega)$ with $supp \ g\subset\Gamma$, then
\[
q_{1}=q_{2}.
\]
\end{lem}
It is well known that $W^{1,\gamma}(\Omega)(\gamma>2)$ can be imbedded into $L^{\infty}(\Omega)$ \cite{sobolev space}. Thus in this paper,  we discuss the problem  \emph{IP} in a larger space $L^{\infty}(\Omega)$, which is more appropriate in practice. In reality, we usually do not have the complete  knowledge of the D-N map $\Lambda(q)$. Instead, we can do  a set of $N$ experiments such that  we can define an excitation pattern $g\in H^{\frac{3}{2}}(\partial\Omega)$ and measure the
resulting flux $\frac{\partial u}{\partial \nu}|_{\Gamma \times (0,T]}$  at discrete locations $x\in\Gamma$ using  probes along the accessible part of boundary $\Gamma$.
Thus, the  realistic  \emph{IP} is to identify  $q$ from partial and noisy knowledge of the D-N map.

\section{The continuity of $F(q,g)$ with respect of $q$}
In this section, we investigate  the continuity of the forward operator $F$ with respect to $q$ for a fixed $g\in H^{\frac{3}{2}}(\partial\Omega)$.
By Theorem 2.2 in \cite{paper-unique}, we can obtain the Laplace transform of $u(x,t)$ in Eq.(\ref{model-fpde}) with respect to $t$ as follows
\begin{eqnarray}\label{laplace-model}
\begin{cases}
-\Delta\tilde{u}(x,s)+(q(x)+s^{\alpha})\tilde{u}(x,s)=0  \ \ \text{in} \ \  \Omega, \ \ s>C_{1}, \\
\tilde{u}(x,s)=\tilde{\lambda}(s)g(x) \ \ \text{on} \ \ \partial\Omega, \ \ s>C_{1},
\end{cases}
\end{eqnarray}
where $C_{1}$ is a constant depending only on $\lambda$,  and
\[
\tilde{u}(x,s)=(Lu)(x,s)=\int_0^\infty u(x,t)e^{-st}dt\]
 is the Laplace transform of $u(x,t)$ in $t$ for each fixed $x \in \bar{\Omega}$.
For arbitrarily fixed $s>C_{1}$, we set $C_{0}=s^{\alpha}>0$.
Then by multiplying  the first equation by $\tilde{w} \in H^{1}(\Omega)$ in (\ref{laplace-model}),
 we obtain
\[
-\int_{\Omega}\Delta\tilde{u}\cdot\tilde{w}+
\int_{\Omega}(q(x)+C_{0})\tilde{u}\cdot\tilde{w}=0.
\]
Using the integration by parts formula implies
\[
\int_{\Omega}\triangledown\tilde{u}\cdot\triangledown\tilde{w}-
\int_{\partial\Omega}\tilde{w}\cdot\frac{\partial\tilde{u}}{\partial \nu}+
\int_{\Omega}(q(x)+C_{0})\tilde{u}\cdot\tilde{w}=0.
\]

Let $\langle\cdot,\cdot\rangle$ be  the dual-inner product on $H^{-\frac{1}{2}}(\partial \Omega)\times H^{\frac{1}{2}}(\partial \Omega)$ and $(\cdot,\cdot)$ be the $L^{2}$ inner product.
We define
\begin{eqnarray}\label{laplace-weak-form}
\begin{cases}
\begin{aligned}
a(\tilde{u},\tilde{w})&=\int_{\Omega}\triangledown\tilde{u}\cdot\triangledown\tilde{w}dx+
\int_{\Omega}(q(x)+C_{0})\tilde{u}\cdot\tilde{w}dx,\\
l(\tilde{w})&=\int_{\partial\Omega}\tilde{w}\cdot\frac{\partial\tilde{u}}{\partial \nu}=\langle \tilde{\Lambda}(q)g,\Upsilon_{\partial\Omega}^{D}\tilde{w}\rangle,
\end{aligned}
\end{cases}
\end{eqnarray}
where $\tilde{\Lambda}(q):g\mapsto\frac{\partial\tilde{u}}{\partial \nu}|_{\partial\Omega}$, and $\Upsilon_{\partial\Omega}^{D}$ is the trace operator on $H^{1}(\Omega)$.

For each $q\in Q_{AD}$, let
\[
\triangle Q_{AD}(q):=\{\delta q \in L^{\infty}(\Omega)|\delta q \geq 0, q+\delta q \in Q_{AD}\},
\]
and
\[
\triangle Q_{AD}:=\bigcup\limits_{q\in Q_{AD}}\triangle Q_{AD}(q)=\{\delta q \in L^{\infty}(\Omega)|\delta q \geq 0, \delta q=q-\bar{q}\ \ \text{for some} \ \ q\in Q_{AD}\ \ and \ \ \bar{q}\in Q_{AD}\}.
\]
Thus, $\|\delta q\|_{L^{\infty}(\Omega)}\leq k_{1}$ for all $\delta q \in \triangle Q_{AD}$.
For any $q \in Q_{AD}, \tilde{u}\in H^{1}(\Omega)$, and all $\tilde{w}\in H^{1}(\Omega)$, we have
\begin{equation}
\begin{aligned}
|a(\tilde{u},\tilde{w})|&=|\int_{\Omega}\triangledown\tilde{u}\cdot\triangledown\tilde{w}dx+
\int_{\Omega}(q(x)+C_{0})\tilde{u}\cdot\tilde{w}dx|\\
&\leq \int_{\Omega}|\triangledown\tilde{u}\cdot\triangledown\tilde{w}|dx+
\int_{\Omega}(\|q\|_{L^{\infty}(\Omega)}+C_{0}) |\tilde{u} \tilde{w}|dx\\
&\leq \int_{\Omega}|\triangledown\tilde{u}|\cdot|\triangledown\tilde{w}|dx+
\int_{\Omega}(\|q\|_{L^{\infty}(\Omega)}+C_{0}) |\tilde{u}| |\tilde{w}|dx\\
&\leq \max\{1,\|q\|_{L^{\infty}(\Omega)}+C_{0}\}\big(\|\triangledown\tilde{u}\|_{L^{2}(\Omega)}
\|\triangledown\tilde{w}\|_{L^{2}(\Omega)}+
\|\tilde{u}\|_{L^{2}(\Omega)}\|\tilde{w}\|_{L^{2}(\Omega)}\big)\\
&\leq C_{2}\|\tilde{u}\|_{H^{1}(\Omega)}\|\tilde{w}\|_{H^{1}(\Omega)},
\end{aligned}
\end{equation}
where $C_{2}=\max\{1,\|q\|_{L^{\infty}(\Omega)}+C_{0}\}$,

and
\begin{equation}\label{}
\begin{aligned}
a(\tilde{w},\tilde{w})&=\int_{\Omega}\triangledown\tilde{w}\cdot\triangledown\tilde{w}dx+
\int_{\Omega}(q(x)+C_{0})\tilde{w}\cdot\tilde{w}dx\\
&\geq \int_{\Omega}\triangledown\tilde{w}\cdot\triangledown\tilde{w}dx+
(1/k_{2}+C_{0}) \int_{\Omega}\tilde{w}\cdot \tilde{w}dx\\
&\geq \min\{1,1/k_{2}+C_{0}\}\int_{\Omega}(\triangledown\tilde{w}\cdot\triangledown\tilde{w}+
\tilde{w} \cdot \tilde{w})dx\\
&= 1/C_{3}\|\tilde{w}\|^{2}_{H^{1}(\Omega)}.
\end{aligned}
\end{equation}
where $C_{3}=1/\min\{1,1/k_{2}+C_{0}\}$.

Let $L(H^{1}(\Omega),(H^{1}(\Omega))^*)$  be the space of bounded linear operator form $H^{1}(\Omega)$ to  $(H^{1}(\Omega))^*$.
  By Lax-Milgram Theorem, for each $q\in Q_{AD}$, there exists a unique bounded linear operator
$A(q)\in L(H^{1}(\Omega),(H^{1}(\Omega))^*)$
satisfying
\[
 (A(q)\tilde{u},\tilde{w})=\langle\tilde{\Lambda}(q)g,\Upsilon_{\partial\Omega}^{D}\tilde{w}\rangle
\]
such that

\begin{equation}
\begin{aligned}
\|A(\delta q)\|_{L(H^{1}(\Omega),(H^{1}(\Omega))^*)}&\leq C_{2},\\
\|A^{-1}(\delta q)\|_{L((H^{1}(\Omega))^*,(H^{1}(\Omega)))}&\leq C_{3},\\
\|A(q)\|_{L(H^{1}(\Omega),(H^{1}(\Omega))^*)}&\leq C_{2},\\
\end{aligned}
\end{equation}

Thus by setting $\tilde{F}(q):=\tilde{u}(x,s)$ , it is easy to show
\[
\tilde{F}(q)=A^{-1}(q)(\Upsilon_{\partial\Omega}^{D})^*\tilde{\Lambda}(q)g.
\]
Now we discuss the regularity of the forward operator.

Notice that $\|\delta q\|_{L^{\infty}(\Omega)}\leq k_{1}$, similar to the proof of Lemma 3.1 in \cite{sup-L^{1}}, we can get the following lemma.
\begin{lem}\label{lap-lemma1}
 Let $W$ be a compact subset of
 $H^{1}(\Omega)$.
  Then for $\delta q \in\triangle Q_{AD}$,\\
 \[
 \sup_{w\in W}\left\{\int_{\Omega}|\delta q(x)||w(x)|^{2}dx\right\}\rightarrow0 \ \ as \ \ \|\delta q\|_{L^{1}(\Omega)}\rightarrow0.
 \]
\end{lem}
The following theorem shows that the map $\tilde{F}$ is continuous.
\begin{thm}\label{F-conti}
The map $\tilde{F}$ from $Q_{AD}$ to
$H^{1}(\Omega)$ is continuous.
\end{thm}

\begin{proof}
Let $\bar{q} \in Q_{AD}$ be fixed and $W$ be a compact set of $(H^{1}(\Omega))^*$.
 Then for $\delta q \in\triangle Q_{AD}(\bar{q})$, let $q := \bar{q}+\delta q$ and for any $w \in W\subset (H^{1}(\Omega))^*$,
  we define
\[
\bar{u}_{w}:= A^{-1}(\bar{q})w,\ u_{w}:= A^{-1}(q)w, \ \delta u_{w}:=\bar{u}_{w}-u_{w}.
\]
By  the coerciveness of the bilinear form $a(\cdot, \cdot)$, we have
\begin{equation}
\begin{aligned}
1/C_{3}\|\delta u_{w}\|^{2}_{H^{1}(\Omega)}\leq& |a(\delta u_{w},\delta u_{w})|\\
=&|( \triangledown \delta u_{w},\triangledown \delta u_{w} )+((q(x)+C_{0})\delta u_{w},\delta u_{w})|\\
=&|(\triangledown \bar{u}_{w},\triangledown\delta u_{w})-(\triangledown u_{w},\triangledown\delta u_{w})+\\
&((q(x)+C_{0})\bar{u}_{w},\delta u_{w})-((q(x)+C_{0})u_{w},\delta u_{w})|\\
=&|(\triangledown \bar{u}_{w},\triangledown\delta u_{w})-(\triangledown u_{w},\triangledown\delta u_{w})+\\
&((\bar{q}(x)+C_{0})\bar{u}_{w},\delta u_{w})-((q(x)+C_{0})u_{w},\delta u_{w})+
(\delta q \bar{u}_{w},\delta u_{w})|\\
\leq &|(\delta q \bar{u}_{w},\delta u_{w})|.
\end{aligned}
\end{equation}

By the Cauchy-Schwarz inequality, we can obtain
\begin{equation}
\begin{aligned}
\|\delta u_{w}\|_{H^{1}(\Omega)}^{2}&\leq C_{3}|(\delta q \bar{u}_{w},\delta u_{w})|\\
&\leq C_{3} \|\delta q \bar{u}_{w}\|_{L^2(\Omega)} \|\delta u_{w}\|_{L^2(\Omega)}\\
&\leq C_{3} \|\delta q \bar{u}_{w}\|_{L^2(\Omega)} \|\delta u_{w}\|_{H^{1}(\Omega)}.
\end{aligned}
\end{equation}
Thus,
\begin{equation}
\begin{aligned}
\|\delta u_{w}\|_{H^{1}(\Omega)}&\leq C_{3}\|\delta q \bar{u}_{w}\|_{L^2(\Omega)}\\
&\leq C_{3}\big(\int_{\Omega}|\delta q|^{2}|\bar{u}_{w}|^{2}dx\big)^{\frac{1}{2}}\\
&\leq C_{3}\big(\int_{\Omega}k_{1}|\delta q||\bar{u}_{w}|^{2}dx\big)^{\frac{1}{2}}.
\end{aligned}
\end{equation}

Since $A^{-1}(\bar{q})$ is a continuous linear map from $(H^{1}(\Omega))^*$ to $H^{1}(\Omega)$ and $\bar{u}_{w}=A^{-1}(\bar{q})w$ with $\omega$ in the compact subset $W\subset (H^{1}(\Omega))^*$, we have that $R:=\{\bar{u}_{w}:w\in W\}$ is a compact subset of $H^{1}(\Omega)$. Hence, by Lemma \ref{lap-lemma1}, it follows that
\[
\|A^{-1}(\bar{q}+\delta q)w-A^{-1}(\bar{q})w\|_{H^{1}(\Omega)}\rightarrow0\ \ as \ \ \|\delta q\|_{L^{1}(\Omega)}\rightarrow0.
\]
Let $W$ be the singleton set $\{(\Upsilon_{\partial\Omega}^{D})^*\tilde{\Lambda}(\bar{q})g\}\subset (H^{1}(\Omega))^*$. Because of the continuity of the $\tilde{\Lambda}(\bar{q})$ with respect to $\bar{q}$ \cite{N.Mandache-2001}, the continuity of the map $\tilde{F}$ from $Q_{AD}$ to $H^{1}(\Omega)$ as $\|\delta q\|_{L^{\infty}(\Omega)}\rightarrow0$ can be obtained immediately.
\end{proof}


\section{The $L^{2}+BV$ regularization method}
We use $f$ and $f^{\delta}$ to represent the exact data and measurement data with noise, respectively,  such that
\[
\|f^{\delta}-f\|_{L^{2}(0,T;L^{2}(\Gamma))}\leq \delta,
\]
where $\delta$ is the noise level.

We assume that $N$ experiments have been conducted with boundary excitation $g_i\in H^{\frac{3}{2}}(\partial\Omega)$,
where the corresponding flux $f_i$ at $\Gamma$ is  measured. Thus, the goal  of realistic \emph{IP} is to find the coefficient $q$ such that $\mathcal{F}(q,g_i)=f_i$, $i=1,2,\ldots,N$, where $\mathcal{F}$ is the operator defined by
\begin{equation}\label{exact-inverse-operator}
\mathcal{F}(q,g_i)=\Upsilon_{\Gamma}^{N}F(q,g_i),
\end{equation}
where $\Upsilon_{\Gamma}^{N}$ and $F$ are  the Neumann trace operator and forward operator defined in section 2, respectively.

In order to recover the coefficient $q$ for different cases,
different regularization methods may be required. To present the regularization method,
 we introduce some function spaces.
The total variation of a function $g\in L^{1}(\Omega)$ is defined by \cite{CR.Vogel-2002}
\[
TV(g)=\sup_{\vec{v}\in\mathcal{V}}\int_{\Omega} g \text { div} \  \vec{v}\ dx,
\]
where the space of test functions
\[
\mathcal{V}=\{\vec{v}\in C_{0}^{1}(\Omega;\mathds{R}^{d})\ | \ |\vec{v}(x)|\leq1 \ \ \text{for all} \ x\in \Omega\}.
\]

The space of functions of bounded variation \cite{CR.Vogel-2002}, denoted by $BV(\Omega)$, consists of functions $g\in L^{1}(\Omega)$ for which
\[
\|g\|_{BV}:=\|g\|_{L^{1}(\Omega)}+TV(g)<\infty.
\]
Now we redefine the set  in (\ref{admissible-set}) by
\[
Q_{AD}:=\{q \in L^{\infty}(\Omega)\cap BV(\Omega), 1/k_{2}\leq q(x)\leq k_{1}\}.
\]

It is known that  $L^{2}$ regularization method is used to reconstruct  smooth functions,
and  $BV$ regularization method is used to reconstruct functions with jump discontinuities.
 Further, the multi-parameter regularization method $L^{2}+BV$  is used to recover the piecewise smooth functions.
For generalization, we consider  these regularization methods in a unified  variational form with different regularization parameters.
To this end, we define the following regularization functional
\begin{equation}\label{regular-functional}
J_{\beta,\gamma}(q;f^{\delta})=\frac{1}{2}\sum_{i=1}^N\|\mathcal{F}(q,g_i)-f_i^{\delta}\|_{L^{2}(0, T;L^{2}(\Gamma))}^{2}+\beta\|q\|_{L^{2}(\Omega)}^{2}
+\gamma\|q\|_{BV(\Omega)},
\end{equation}
where $\beta$, and $\gamma$ are regularization parameters.

Next, we present the result about the existence and stability of the regularized  variational problem.  Without loss of generality, we set $N=1$ for the following analysis, and denote $\mathcal{F}(q,g_i)$ by $\mathcal{F}(q)$.
\begin{thm}
The regularization functional $J_{\beta,\gamma}(q;f^{\delta})$ exists a minimizer in $Q_{AD}$ and the minimizers of $J_{\beta,\gamma}(q;f^{\delta})$ are stable with respect to the measurement data $f^{\delta}$, i.e., let every $q_{k}$ satisfing
\[
q_{k}=\min_{q\in Q_{AD}} J_{\beta,\gamma}(q;f_{k}),\ \  \text{where}\ \ f_{k}\xrightarrow{L^{2}(0,T;L^{2}(\Gamma))} f^{\delta} \ \ \text{as}\ \ k\rightarrow \infty,
\]
then $\{q_{k}\}$ has a subsequence $\{q_{k_{j}}\}$ such that
\[
q_{k_{j}} \xrightarrow{L^{1}(\Omega)} q^{*}\ \ \text{as}\ \ k\rightarrow \infty,
\]
where $q^{*}$ is a minimizer of $ J_{\beta,\gamma}(q;f^{\delta})$.
\end{thm}

\begin{proof}
From Theorem \ref{F-conti}, we see that
\[
\mathcal{F}: Q_{AD}\subseteq L^{1}(\Omega)\mapsto L^{2}(0, T;L^{2}(\Gamma))
\]
is continuous.
Meanwhile, because of
 the compact embedding of $BV(\Omega)$ in $L^{1}(\Omega)$, strictly convex and lower semi-continuity of penalty term,  Theorem 3.1 and Theorem 3.2 in \cite{Tikhonov-conver} implies  the existence and stability of minimizer of $J_{\beta,\gamma}(q;f^{\delta})$.
\end{proof}

Using the techniques in \cite{Tikhonov-conver}, we can get the convergence result of the regularization method.
\begin{thm}
Assume that the sequence $\{\delta_{k}\}$ converges monotonically to $0$ and the corresponding measurement data $f_{k}^{\delta}:=f^{\delta_{k}}$
satisfies
$\|f_{k}^{\delta}-f\|_{L^{2}(0,T;L^{2}(\Gamma))}\leq \delta_{k}$.
Furthermore, suppose that the regularization parameter $\beta(\delta)$ and $\gamma(\delta)$  monotonically increase and satisfy
\[
\beta(\delta)\thickapprox\gamma(\delta),\ \  \beta(\delta)\rightarrow 0 \ \ \text{and}
\ \ \frac{\delta^{2}}{\beta(\delta)}\rightarrow 0 \ \ \text{as}\ \ \delta\rightarrow 0.
\]
Here, $\beta(\delta)\thickapprox\gamma(\delta)$ means $d_{0}\gamma(\delta)\leqslant\beta(\delta)\leqslant d_{1}\gamma(\delta)$ for some positive constant $d_{0}$ and $d_{1}$ independent of the parameters involved.
Setting $\beta_{k}=\beta(\delta_{k}),\gamma_{k}=\gamma(\delta_{k})$, then the corresponding minimizers
\[
q_{k}=\min_{q \in \bar{Q}_{AD}}J_{\beta_{k},\gamma_{k}}(q;f_{k}^{\delta}),
\]
has a subsequence, which converges to the solution of $\mathcal{F}(q)=f$.
\end{thm}

\begin{proof}
From the definition of $q_{k}$, it follows that
\begin{equation}\label{}
\begin{aligned}
&\frac{1}{2}\|\mathcal{F}(q_{k})-f_{k}^{\delta}\|^{2}_{L^{2}(0, T;L^{2}(\Gamma))}+\beta_{k}\|q_{k}\|_{L^{2}(\Omega)}^{2}
+\gamma_{k}\|q_{k}\|_{BV(\Omega)}\\
\leq&\frac{1}{2}\|\mathcal{F}(q)-f_{k}^{\delta}\|^{2}_{L^{2}(0, T;L^{2}(\Gamma))}+\beta_{k}\|q\|_{L^{2}(\Omega)}^{2}
+\gamma_{k}\|q\|_{BV(\Omega)}\\
\leq&\frac{1}{2}\delta_{k}^{2}+\beta_{k}\|q\|_{L^{2}(\Omega)}^{2}
+\gamma_{k}\|q\|_{BV(\Omega)}.
\end{aligned}
\end{equation}
This  implies
\[
\|\mathcal{F}(q_{k})-f_{k}^{\delta}\|_{L^{2}(0, T;L^{2}(\Gamma))}\rightarrow0 \ \ as \ \ k\rightarrow\infty,
\]
and
\[
\beta_{k}\|q_{k}\|_{L^{2}(\Omega)}^{2}+\gamma_{k}\|q_{k}\|_{BV(\Omega)}\leq
\frac{1}{2}\delta_{k}^{2}+\beta_{k}\|q\|_{L^{2}(\Omega)}^{2}
+\gamma_{k}\|q\|_{BV(\Omega)}.
\]
Since $\beta(\delta)\thickapprox\gamma(\delta)$, we get
\[
d_{0}\|q_{k}\|_{L^{2}(\Omega)}^{2}+\|q_{k}\|_{BV(\Omega)}\leq
\frac{1}{2}\frac{\delta_{k}^{2}}{\gamma_{k}}+d_{1}\|q\|_{L^{2}(\Omega)}^{2}+\|q\|_{BV(\Omega)},
\]
and then
\[
\lim_{k\rightarrow\infty}\sup\left(d_{0}\|q_{k}\|_{L^{2}(\Omega)}^{2}+
\|q_{k}\|_{BV(\Omega)}\right)\leq d_{1}\|q\|_{L^{2}(\Omega)}^{2}+\|q\|_{BV(\Omega)}.
\]
Thus, it follows that
\begin{equation}\label{}
\begin{aligned}
&\lim_{k\rightarrow\infty}\sup\left(\|\mathcal{F}(q_{k})-f_{k}^{\delta}\|
_{L^{2}(0, T;L^{2}(\Gamma))}^{2}+\beta_{1}\|q_{k}\|_{L^{2}(\Omega)}^{2}+
\gamma_{1}\|q_{k}\|_{BV(\Omega)}\right)\\
\leq&\lim_{k\rightarrow\infty}\sup\left(\|\mathcal{F}(q_{k})-f_{k}^{\delta}\|
_{L^{2}(0, T;L^{2}(\Gamma))}^{2}+\beta_{k}\|q_{k}\|_{L^{2}(\Omega)}^{2}+
\gamma_{k}\|q_{k}\|_{BV(\Omega)}\right)\\
&+\lim_{k\rightarrow\infty}\sup\left((\beta_{1}-\beta_{k})\|q_{k}\|_{L^{2}(\Omega)}^{2}+
(\gamma_{1}-\gamma_{k})\|q_{k}\|_{BV(\Omega)}\right)\\
\leq&\frac{\beta_{1}}{d_{0}}\lim_{k\rightarrow\infty}\sup\left(d_{0}\|q_{k}\|_{L^{2}(\Omega)}^{2}
+\|q_{k}\|_{BV(\Omega)}\right)\\
\leq&\frac{\beta_{1}}{d_{0}}\left(d_{1}\|q\|_{L^{2}(\Omega)}^{2}+\|q\|_{BV(\Omega)}\right)\\
<&\infty.
\end{aligned}
\end{equation}
Note that $BV(\Omega)$ is a relatively compact subset of $L^{1}(\Omega)$, then $q_{k}$ has a subsequence $\{q_{k_{j}}\}$ and $q_{k_{j}}\xrightarrow{L^{1}(\Omega)} q^{*}$, $q^{*}\in Q_{AD}$ satisfies (\ref{exact-inverse-operator}). The proof is completed.
\end{proof}

\begin{rem}
If we specify  the source condition and the appropriate regularization parameter selection strategy, such as discrepancy principle \cite{H.W.Engl-1996}, the convergence rate of regularization solution with respect to noise can also be obtained.
\end{rem}

\section{Solving the forward problem using mixed finite element method}\label{forward problem}
In the following, the mixed finite element method \cite{forward-mixedfem-2} is presented to solve the forward problem numerically.
Here $\Omega$ is a $(0,1)\times(0,1)$ domain and let $T=1$ be fixed.
To this end, we need to discretize the spatial space and temporal space.
Define $x_{i}=i\Delta x \ (i=0,1,\cdots,N_{1})$, $y_{j}=j\Delta y \ (j=0,1,\cdots,N_{2})$ and $t_{n}=n\Delta t \ (n=0,1,\cdots,M)$, where $\Delta x=1/N_{1}, \Delta y=1/N_{2}$, and $\Delta t=1/M$ are the space and time step sizes, respectively.
Let $\omega=-\nabla u$ in $\Omega$.  Then Eq.(\ref{model-fpde}) turns to
\begin{eqnarray}
\label{model-mixedfem}
\begin{cases}
\omega(x,t)+\nabla u(x,t)=0 \ \ \text{in} \ \  \Omega \times (0,T],\\
{}_0 D_t^\alpha u(x,t)+div(\omega(x,t))+q(x)u(x,t)=0  \ \ \text{in} \ \  \Omega \times (0,T], \\
u(x,0)=0  \ \ \text{in} \ \ \Omega,\\
u(x,t)=\lambda(t)g(x) \ \ \text{on} \ \  \partial \Omega \times (0,T].
\end{cases}
\end{eqnarray}
Now we multiply the first equation in (\ref{model-mixedfem}) by $v, \forall v\in H(div;\Omega)$. Integrating by parts and using the Dirichlet boundary condition for $u$, then we can obtain
\[
\int_{\Omega}\omega\cdot v-\int_{\Omega}div(v)\cdot u=-\int_{\partial\Omega}uv\cdot \nu, \ \ \forall v\in H(div;\Omega).
\]
Then multiplying the second equation in (\ref{model-mixedfem}) by $p, \forall p\in L^{2}(\Omega)$ and integrating it on $x$ over $\Omega$, we get
\[
\int_{\Omega}{}_0 D_t^\alpha u\cdot p+\int_{\Omega}div(\omega)\cdot p+\int_{\Omega}qu\cdot p=0, \ \ \forall p\in L^{2}(\Omega).
\]
Thus our goal is: find ${\omega,u}:[0,T]\mapsto H(div;\Omega)\times L^{2}(\Omega)$ such that
\begin{eqnarray}
\label{weakconti-mixed}
\begin{cases}
\int_{\Omega}\omega\cdot v-\int_{\Omega}div(v)\cdot u=-\int_{\partial\Omega}uv\cdot \nu, \ \ \forall v\in H(div;\Omega),\\
\int_{\Omega}{}_0 D_t^\alpha u\cdot p+\int_{\Omega}div(\omega)\cdot p+\int_{\Omega}qu\cdot p=0, \ \ \forall p\in L^{2}(\Omega).
\end{cases}
\end{eqnarray}
For function $u(x,t)$, denote $u^{n}=u^{n}(\cdot)=u(\cdot,t_{n})$.
For the Caputo derivative ${}_0 D_t^\alpha u $, we use the following approximation \cite{Lin_xu},
\begin{equation}\label{appro-caputo}
  {}_0 D_t^\alpha u^{n}=\frac{1}{(\Delta t)^{\alpha}\Gamma(2-\alpha)}\sum_{k=1}^{n}(u^{k}-u^{k-1})\times[(n+1-k)^{1-\alpha}-(n-k)^{1-\alpha}].
\end{equation}

Suppose  that $\omega_{h}^{n}=\sum_{i=1}^{I}\sigma_{i}^{n}\psi_{i}$ and $u_{h}^{n}=\sum_{k=1}^{J}\beta_{k}^{n}\phi_{k}$,where $\{\psi_{i}\}_{i=1}^I$ and $\{\phi_{k}\}_{k=1}^J$ are the pair of basis functions in a mixed FEM.
Let $v=\psi_{j},\ p=\phi_{l}$, then the discrete mixed problem can be written as: find $(\omega_{h},u_{h})$ such that
\[
\begin{cases}
\int_{\Omega}\left(\sum_{i=1}^{I}\sigma_{i}^{n}\psi_{i}\right)\cdot\psi_{j}
-\int_{\Omega}div(\psi_{j})\cdot\left(\sum_{k=1}^{J}\beta_{k}^{n}\phi_{k}\right)
=-\int_{\partial \Omega}u(x,t_{n})\cdot \left(\psi_{j}\cdot \nu \right),\\
\int_{\Omega}{}_0
D_t^\alpha\left(\sum_{k=1}^{J}\beta_{k}^{n}\phi_{k}\right)\cdot \phi_{l}
+\int_{\Omega}div\left(\sum_{i=1}^{I}\sigma_{i}^{n}\psi_{i}\right)\cdot \phi_{l}
+\int_{\Omega}q\phi_{l} \cdot\left(\sum_{k=1}^{J}\beta_{k}^{n}\phi_{k}\right)=0.
\end{cases}
\]
We can rewrite  above as the following matrix form,

\[
n=1, \ \
\begin{pmatrix}
A & B \\
sB^{T} & C+sD
\end{pmatrix}
\begin{pmatrix}
\sigma^{1} \\
\beta^{1}
\end{pmatrix}
=
\begin{pmatrix}
G(:,1) \\
b(1)\beta^{0}
\end{pmatrix}
;\]

\[
n\geq 2, \ \
\begin{pmatrix}
A & B \\
sB^{T} & C+sD
\end{pmatrix}
\begin{pmatrix}
\sigma^{n} \\
\beta^{n}
\end{pmatrix}
=
\begin{pmatrix}
G(:,n) \\
c_{1}\beta^{n-1}+c_{2}\beta^{n-2}+\cdots+c_{n-1}\beta^{1}+b(n)\beta^{0}
\end{pmatrix}
.\]
Here
\[
\begin{aligned}
&s=(\Delta t)^{\alpha}\Gamma(2-\alpha),\ \ b_{n}=n^{1-\alpha}-(n-1)^{1-\alpha},\ \ 1\leq n \leq M,\\
&c_{k}=2k^{1-\alpha}-(k+1)^{1-\alpha}-(k-1)^{1-\alpha},\ \ 1\leq k \leq M-1,\\
\end{aligned}
\]
and
\[
\begin{aligned}
& A\in \mathds{R}^{I\times I}, A_{ij}=-\int_{\Omega}\psi_{i} \cdot \psi_{j},\ \
B\in \mathds{R}^{I\times J}, B_{ji}=\int_{\Omega}div(\psi_{j})\phi_{i},\\
& G\in \mathds{R}^{I\times M}, G_{ij}=\int_{\partial \Omega}u(x,t_{j})(\psi_{i}\cdot \nu),\ \
C\in \mathds{R}^{J\times J}, C_{ml}=\int_{\Omega}\phi_{m}\phi_{l},\\
& D\in \mathds{R}^{J\times J}, D_{ij}=\int_{\Omega}q(x)\phi_{i}\phi_{j}.
\end{aligned}
\]

\section{Numerical results}
In this section, the L-M algorithm
\cite{A.Doicu-2010,M.Hanke-1997} is applied to recover the unknown function $q$ in Eq. (\ref{model-fpde}).
Since $q$ is a spatial function,  the dimension of $q$ usually depends on the spatial grid size if no priori
information is available.  Thus the dimension of $q$ may be very high if  the grid number is large.
To overcome the difficulty,  we parameterize the function using some priori information
such that the dimension of $q$ is much less than the grid size of the spatial discretization.
To this end, we choose a set of basis functions to represent the reaction coefficient $q$.
In other words, the basis functions act as some priori information for $q$.
The more information we know, the better inversion results we can obtain.
In Subsection \ref{1}, we use $L^2$ regularization method for recovering a smooth function.
In Subsection \ref{2}, the $BV$ regularization method is applied to recover a piecewise constant function.
In Subsection \ref{3}, the $L^2+BV$ penalty is used to reconstruct a piecewise smooth function.
Moreover, we choose different basis functions for these unknown parameters and use a vector $ a\in \mathds{R}^{n_{q}}$  to represent the iterative solution under dimensional reduction.
For simulation, we set
\[
\lambda(t)=t^2
\]
for Eq.(\ref{model-fpde}). We also choose the numerical differential step  $\tau=0.5$, convergence precision $eps=1\times 10^{-4}$
and the time step size $\Delta t=1/100$, respectively.  We use  the lowest Raviart-Thomas finite element method (a mixed FEM)
to solve the forward problem in the unit square $[0, 1]^2$.
The accuracy of the inversion  solution  is measured by the relative error
\[
\bm{\varepsilon}=\frac{\|q^{inv}-q_{true}\|_{L^2(\Omega)}}
{\|q_{true}\|_{L^2(\Omega)}},
\]
where $q_{true}$ and $q^{inv}$ represent the true and inversion solution, respectively.

The  data $f^{\delta}=[f_{1}^{\delta};\cdots;f_{N}^{\delta}]$ are generated by
\[
f_{i}^{\delta}=\mathcal{F}(q_{true},g_{i})|_{\partial\Omega}+\eta,\quad 1\leq i \leq N,
\]
where $\eta$ is the Gaussian random vector with the  distribution $N(0,\delta^2I)$.
The discretization of spatial space and temporal space are  the same as the forward problem in Section \ref{forward problem}.
The L-M algorithm is presented in Table \ref{algorithm} to reconstruct the unknown coefficient $q$ through solving the minimization problem of (\ref{regular-functional}).

\begin{table}[htbp]
\centering
 \caption{Levenberg-Marquardt algorithm}\label{algorithm}
 \begin{tabular}{l}
  \toprule
 \textbf{Input}: The regularization parameter $\beta,\gamma$,
  the noise level $\delta$ and
  the maximum \\iterations $R$.\\
  \textbf{Output}: $ a_{k+1}=(a^{1}_{k+1},...,a^{n_{q}}_{k+1})^{T}.$\\
  \midrule
  1.  $\bf{Initialize}$ the solution $a$;\\
  2.  Settle the forward problem and obtain the additional data
  $f^{\delta}$;

  \\
  3.  $\bf{While}$ $k<R$\\
  4.  Compute the forward problem and set $F_{k}=f^{\delta}-[\mathcal{F}(q_{k},g_{1});\cdots;\mathcal{F}(q_{k},g_{N})];$\\
  5.  Compute the Jacobian matrix $G_{0}$;\\
  6.  Compute
   $h_{k}=(G_{0}^{T}G_{0}+\beta I+\gamma(L_{1}+L_{2}))^{-1}(G_{0}^{T}F_{k})$,
   where $L_{1}$ and $L_{2}$ are discrete\\ forms of the gradient for the norm of $L^{1}$ and $TV$, respectively;\\
  7.  ${\bf Update}$ the solution
   $a_{k+1}=a_{k}+h_{k}$;\\
  8.  $k\leftarrow k+1$;\\
  9.  $\bf{Terminate}$ if $\|h_{k}\|<eps$;\\
  10.  $\bf{end}$\\
  \bottomrule
 \end{tabular}
 \end{table}

\begin{rem}
In this paper, the Jacobian matrix $G_{0}$ is approximated by a finite difference method, i.e., each column
of $G_{0}$ is computed by
\[
G_{0}(:,j)=\frac{\mathcal{F}(a+\tau\zeta_{j})-
\mathcal{F}(a)}{\tau}
\]
for $j=1,2,\cdots, n_{q}$, where $\tau$ is the step size and $\zeta_{j}$ is a vector of zeros with one in the j-th component. For $L_{1}$ and $L_{2}$, we can see \cite{CR.Vogel-2002}.
\end{rem}


\subsection{Inversion for smooth function}\label{1}
In this subsection, we take
\[
q(x,y)=\cos(\pi x)\sin(\pi y)+1.5
\]
as the true coefficient $q$ in Eq.(\ref{model-fpde}).

For numerical implementation, we  represent the coefficient $q$ in a finite-dimensional space. Then the \emph{IP} turns to search for a vector $q^{inv}\in \mathds{R}^{m}$ such that
\[
q(x,y)=\sum_{i=1}^{m}q^{inv}_{i}\xi_{i}(x,y),
\]
where $\{\xi_{i}\}_{i=1}^{i=m}$ is a set of hat functions and $m$ is the number of elements under discretization. However, the dimension of $q^{inv}$ will increase dramatically if the partition is fine. Assume a priori information we know is that $q$ is smooth. Thus Karhunen-Lo$\grave{e}$ve expansion (KLE) \cite{kle} approach can be  employed to reduce the dimension of $q^{inv}$.

Since $q>0$, we assume that $\log q(x,w)$ is a random filed which can be represented as
\[
\log q(x,\omega)=E[q(x,\omega)]+\sum_{i=1}^{\infty}\sqrt{\lambda_{i}}\varphi_{i}(x)Q_{i}(\omega),
\]
where $\lambda_{i}$ and $\varphi_{i}(x)$ are the eigen-pairs  such that
\[
\int_{\Omega}C(x;x')\varphi_{i}(x')=\lambda_{i}\varphi_{i}(x), \ \
E(Q_{i})=0, \ \
E(Q_{i}(\omega)Q_{j}(\omega))=\delta_{ij}.
\]
In particular, we consider the covariance function
\[
C(x,y;x',y')=\rho^2\exp\left(-\frac{|x-x'|^2}{2l_{1}^2}-\frac{|y-y'|^2}{2l_{2}^2}\right),\ \ (x,y)\in \Omega,
\]
where $\rho^2=0.01$, $l_{1}=l_{2}=0.3$.
For numerical simulation, we use the truncated expansion with the first $n_{q}$-terms
as
\[
\log q(x,\omega)\approx E[q(x,\omega)]+\sum_{i=1}^{n_{q}}\sqrt{\lambda_{i}}\varphi_{i}(x)Q_{i}(\omega),
\]
for $n_{q}\ll m$,
where the choice of $n_{q}$  depends  on the decay rate of the eigenvalues $\lambda_{i}$.
If we choose $E[q(x,\omega)]=0$, we can rewrite $q^{inv}$ in a matrix form as
\[
q^{inv}=\exp(Ha).
\]
]
Here, $a \in\mathds{R}^{n_{q}\times 1}$ and $H\in\mathds{R}^{m\times n_{q}}$, which is defined as
\[
H=[\sqrt{\lambda_{1}}\varphi_{1},
\sqrt{\lambda_{2}}\varphi_{2},\cdots,
\sqrt{\lambda_{n_{q}}}\varphi_{n_{q}}].
\]
To capture 95$\%$ of the energy of the log-normal process, we retain $n_{q}=8$ Karhunen-Lo$\grave{e}$ve modes.
Then, we only need to recover the vector $a$ in a low dimension space.
The forward problem is solved on a uniform $20\times 20$ grid through the  mixed finite element method.
Since $q$ is smooth, $L^{2}$ penalty is used for the example. That is, we set $\gamma=0$ in (\ref{regular-functional}).

By  Lemma \ref{lem-uniqueness}, all possible  $g\in H^{\frac{3}{2}}(\partial\Omega)$ should be included in the Cauchy data to ensure  the uniqueness of the \emph{IP}.
 To check  the effect of data on the inversion solution, we use different number of basis functions in $H^{\frac{3}{2}}(\partial\Omega)$ to generate Cauchy data.
 To this end, we set the boundary function $g(x,y)$ in the form
 \[
g(x, y)=\sin(k_{1}\pi x)\cos(k_{2}\pi y),
\]
where $(k_{1},k_{2})$ is chosen from the following set
\[
H=\big\{(1,1),(1,2),(2,1),(2,2),(1,3)\big\}.
\]
We choose  the first $N=1, 3, 5$  basis functions to generate Cauchy data  for numerical  simulation.
We note that $N$ is the number of experiments for measurement.
The measurement data are taken at time $t=[61, 71, 81, 91, 101]\cdot\Delta t$.
The relative $L^2$ errors versus iterative steps in L-M algorithm  are plotted in
Fig \ref{relative-data} when different numbers of measurement experiments  are used, i.e., $N=1,3,5$. By the figure,
as the number of measurement experiments  increases, the convergence becomes faster.
We can also see that the effect of increasing  data becomes very small when the measurement experiment  number $N=3$. This implies  that the  data is saturated for $N=3$ in this case.

In practical simulation,  we often use few measurement experiments  for inversion to reduce the cost of measurements and computation.
We will choose the experiment  number  $N=1$ to identify the unknown $q$ in the following  numerical computation.

\begin{figure}
\centering
\includegraphics[width=0.50\textwidth, height=0.40\textwidth]{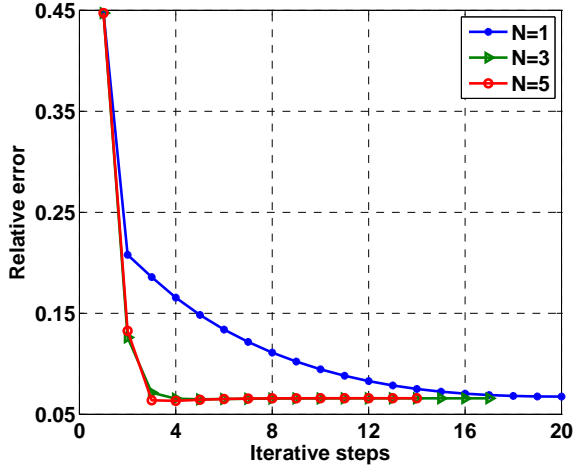}
\caption{Relative errors versus iterative steps in Section 6.1. }\label{relative-data}
\end{figure}

In the case of regularization parameter $\beta=5e-4$, the true solution and the corresponding inversion solution with the noise level $\delta=1\%$ are plotted in Fig \ref{fig-sin-solution}.
From the figure, we can find the inversion solution profile matches   the real solution profile  well. This
 shows that $L^2$ regularization method is suitable  for the smooth function and the choice of basis functions is proper for this example. Under the same conditions, Fig \ref{fig-sin-yfixed} shows the cross sectional drawings of the true and the inversion solutions with $y=0.1$, $y=0.5$, $y=0.9$. It is  clear that the inversion solution approximate the exact solution well with $y=0.1$, $y=0.9$ except for the location where $x$ is near to $0$.
The estimate for $y=0.5$ is also acceptable.
The first plot in Fig \ref{fig-sin-convergence} shows the relative error versus the
iteration number. By the plot, we see
 that the relative error $\bm{\varepsilon}$ decays  as the number of iterations increases.
This shows  that the algorithm is efficient and the choice of $\beta$ is effective  for the example.
To access the impact  of  noise on the  inversion,   we conduct $20$
 experiments under three different noise level $\delta$, and the results are depicted  in the second plot in Fig \ref{fig-sin-convergence}.
It demonstrates  that the relative error becomes more oscillation as  the noise level $\delta$ increases.

\begin{figure}
\centering
\subfigure[]{
\includegraphics[width=0.31\textwidth, height=0.36\textwidth]{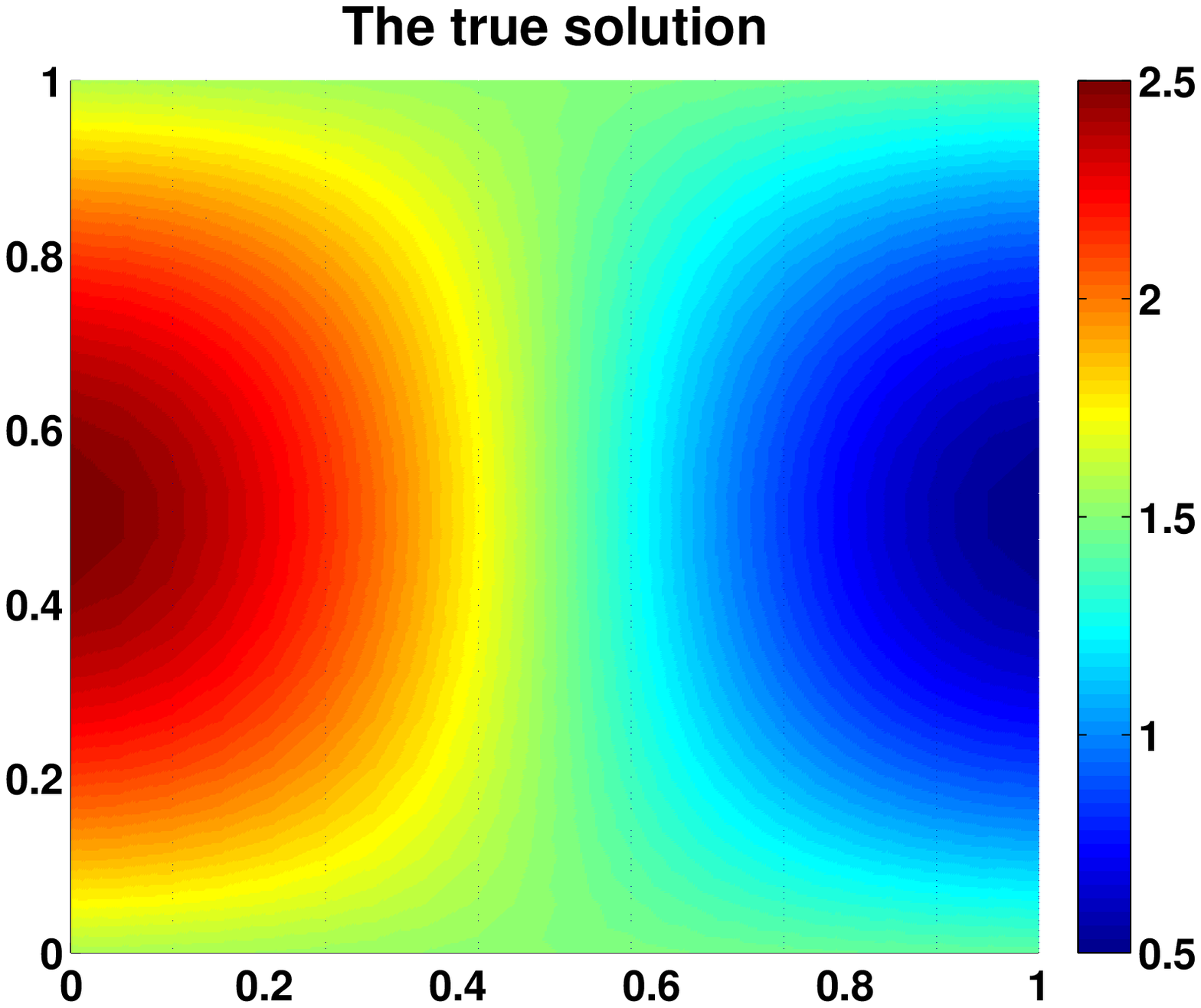}
}
\subfigure[]{
\includegraphics[width=0.31\textwidth, height=0.36\textwidth]{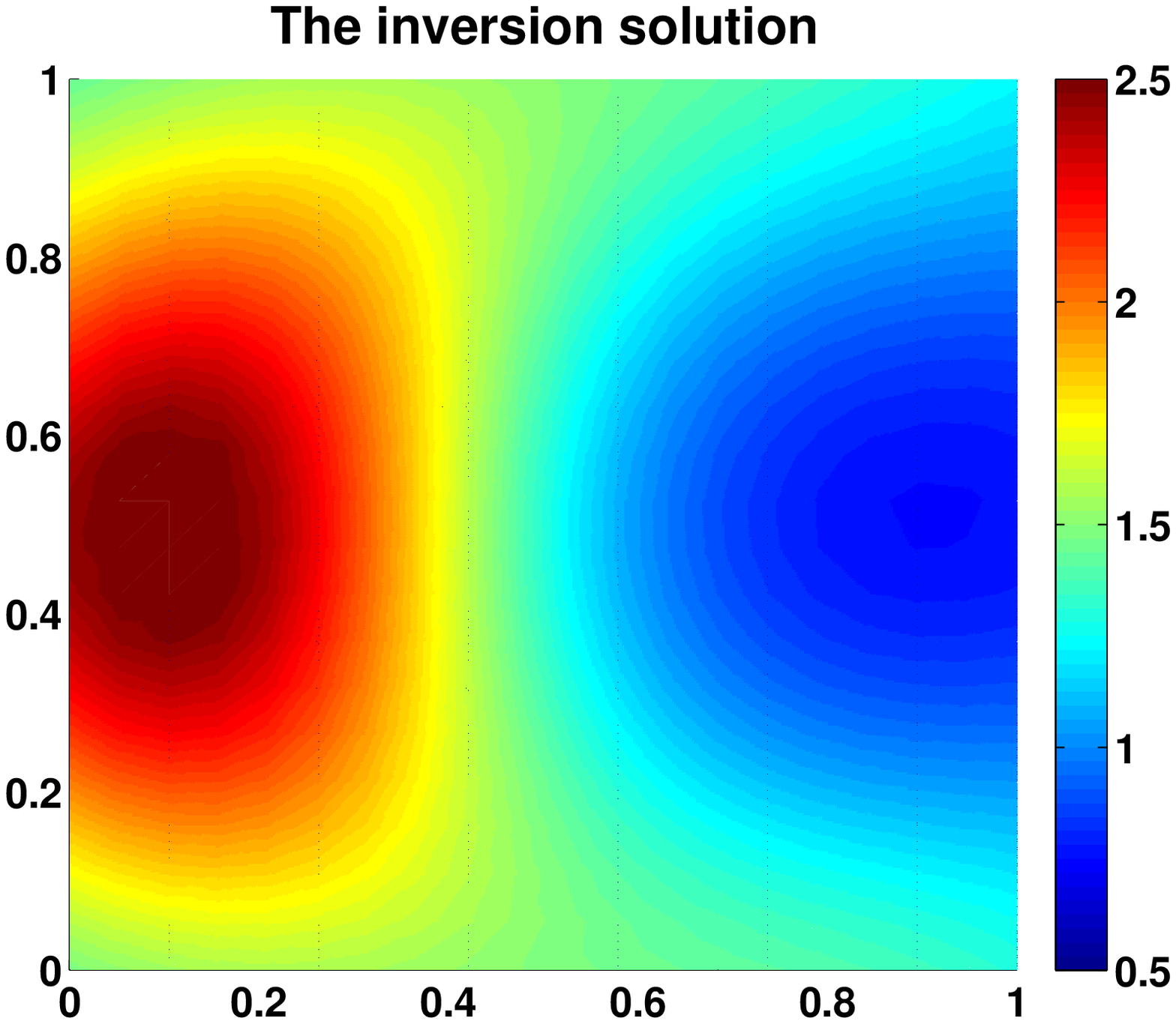}
}
\caption{The true solution and inversion solution in Section 6.1.}\label{fig-sin-solution}
\end{figure}

\begin{figure}
\centering
\subfigure[]{
\includegraphics[width=0.28\textwidth, height=0.28\textwidth]{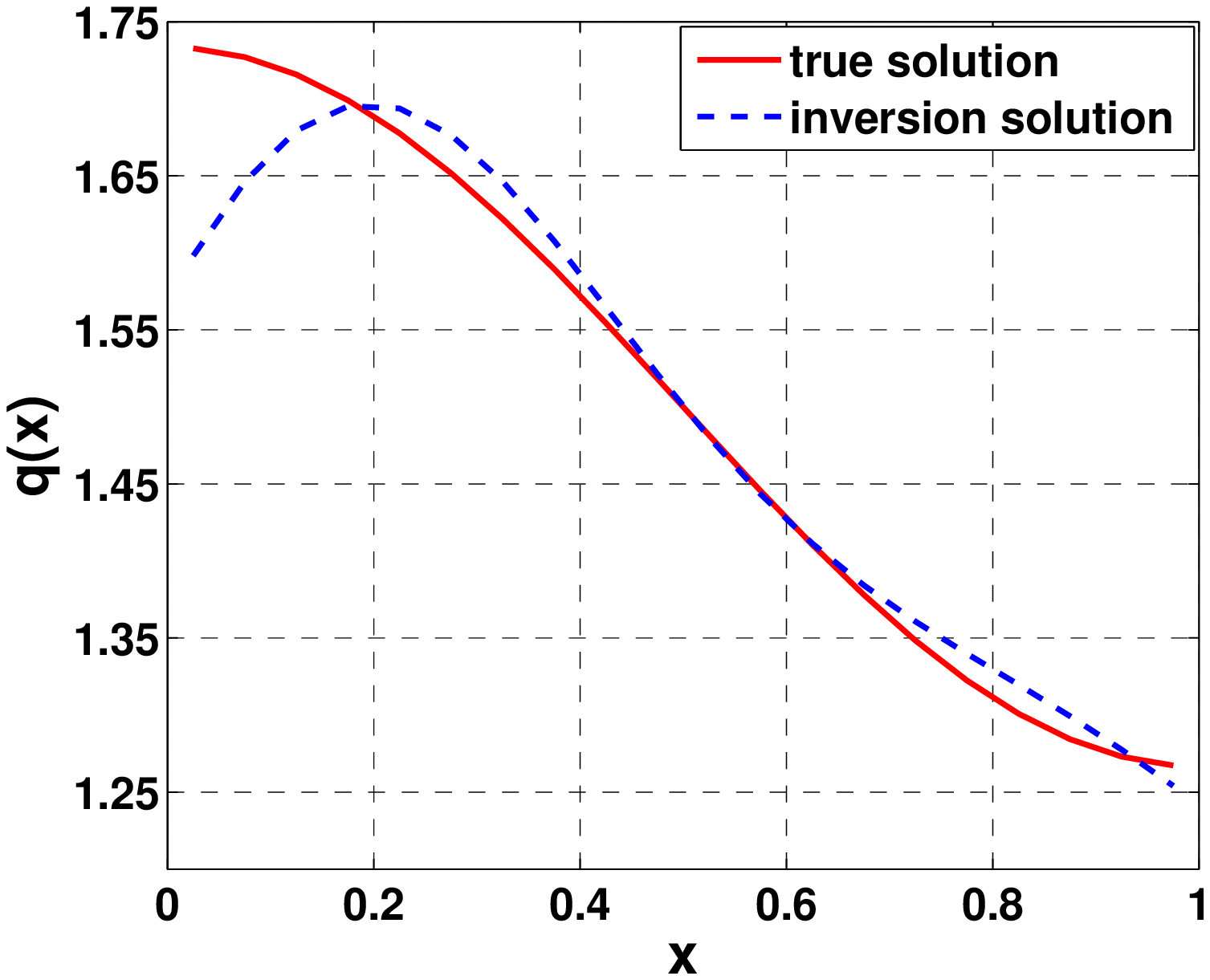}
}
\subfigure[]{
\includegraphics[width=0.28\textwidth, height=0.28\textwidth]{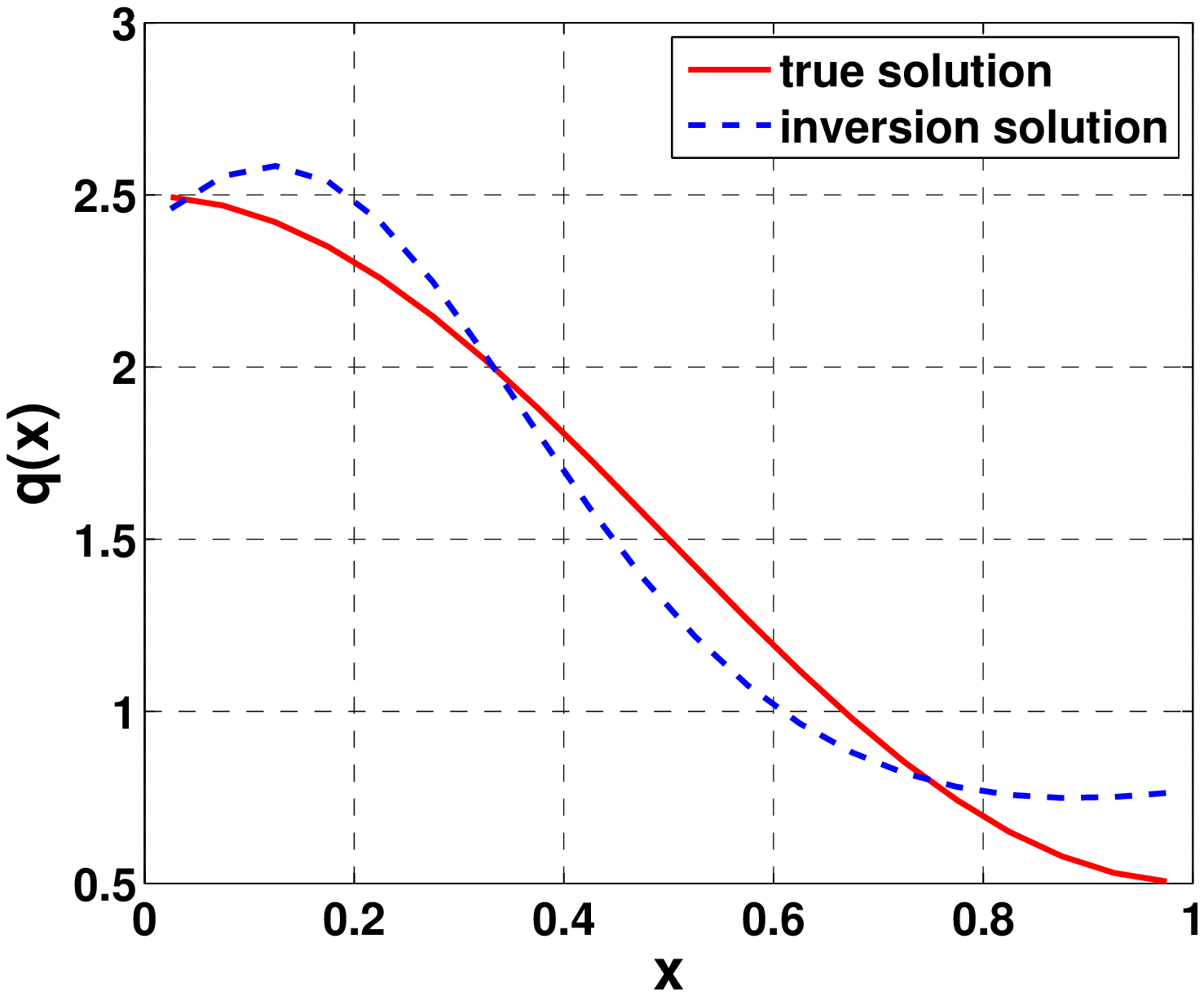}
}
\subfigure[]{
\includegraphics[width=0.28\textwidth, height=0.28\textwidth]{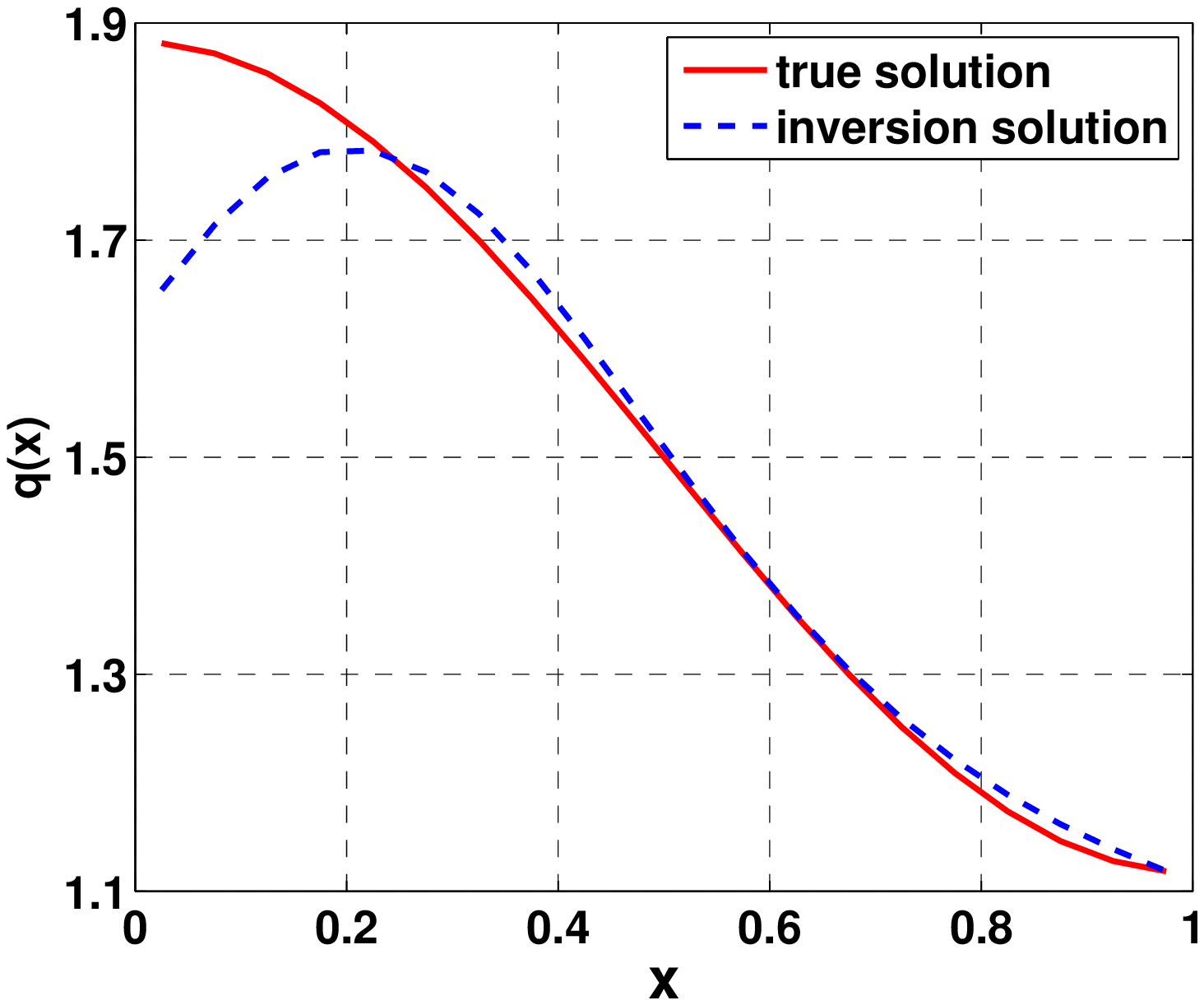}
}
\caption{The true solution and inversion solution with $y=0.1$, $ 0.5$, $0.9$ in Section 6.1}\label{fig-sin-yfixed}
\end{figure}

\begin{figure}
\centering
\subfigure[]{
\includegraphics[width=0.31\textwidth, height=0.36\textwidth]{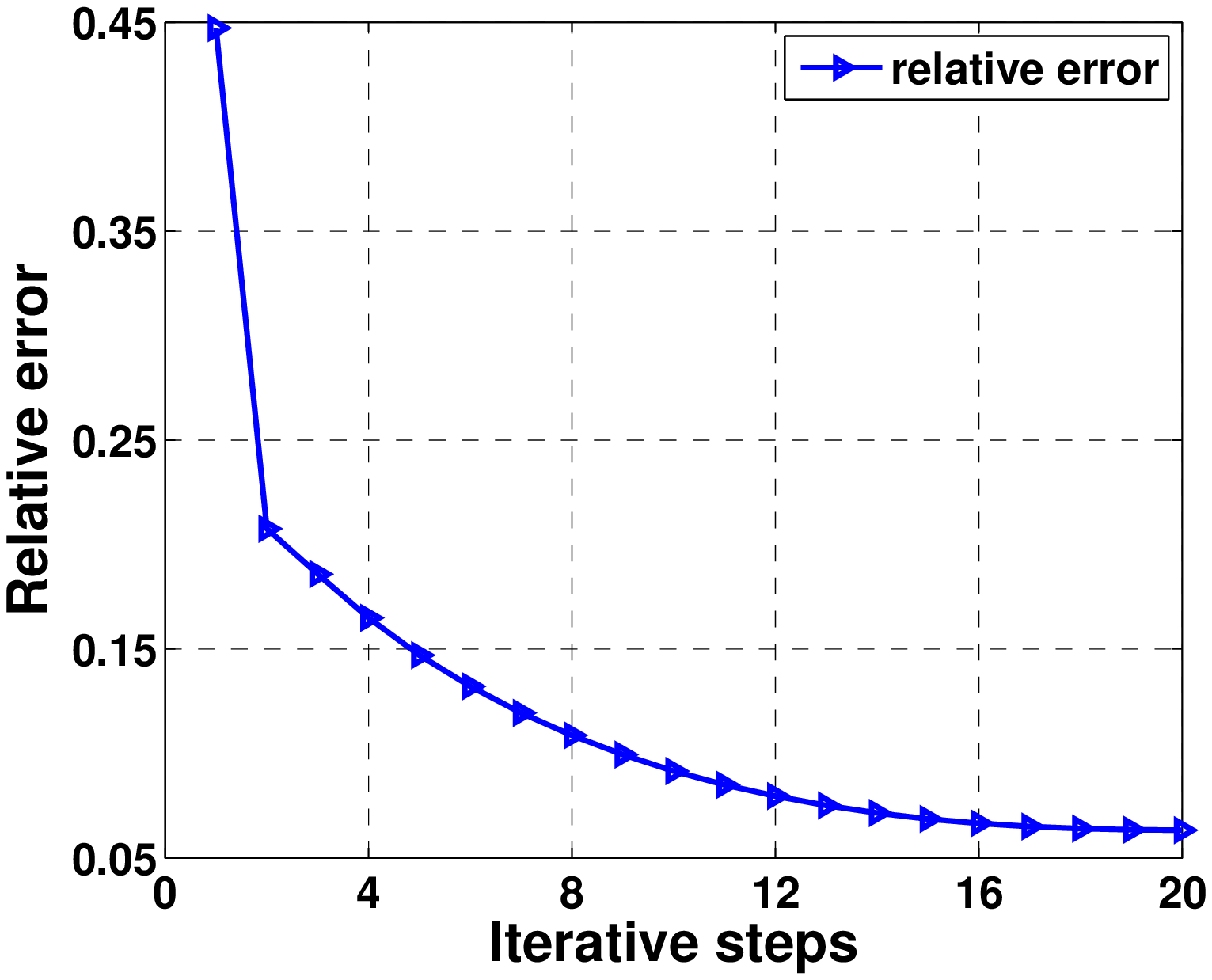}
}
\subfigure[]{
\includegraphics[width=0.31\textwidth, height=0.36\textwidth]{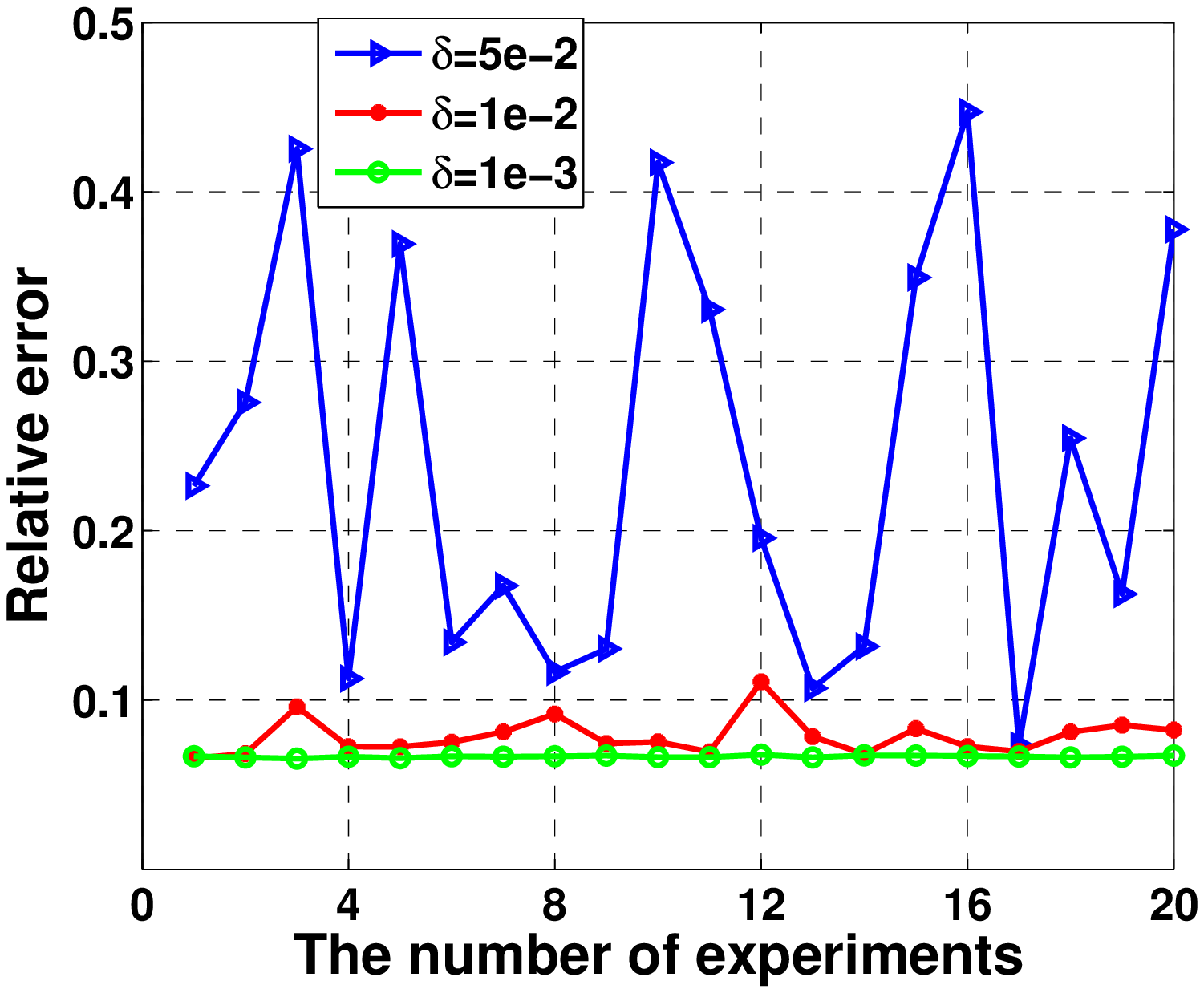}
}
\caption{(a): Relative errors versus iterative steps in Section 6.1, (b): relative errors under 20 experiments with different noise levels in Section 6.1. }\label{fig-sin-convergence}
\end{figure}


\subsection{Inversion for function with jump discontinuities}\label{2}
In this subsection, we consider  the case when $q$ is piecewise constant with jump discontinuities.
We take
\[q(x,y)=\left\{\begin{array}{ll}
10,&\text{$0\leq x \leq 2/3,\
0\leq y \leq 1/3$},\\
1,& else.
\end{array}\right.\]
as the true coefficient $q$.
Let $\Omega$ consist of  $3\times3$ uniform sub-regions and  the value of $q$ on each sub-region be a constant. Then the dimension of $a$ is nine. The $BV$ regularization is used for the example due to the jump discontinuity  of $q$, i.e.,  $\beta=0$ is fixed.

We solve the forward problem on a uniform $18\times 18$ grid. The measurement data are taken at time $t=(1+[60;70;80;90;100])\cdot\Delta t$.
In the case of $\alpha=0.4$ and $\gamma=5e-3$, the true solution and the inversion solution are plotted in Fig \ref{fig-Bv-solution}.
From the figure, we find the inversion result approximate the true solution well.
It shows that the function with jump discontinuities can
 be effectively reconstructed by the $BV$ regularization method.
The left plot in Fig \ref{fig-Bv-relative} illustrates the relative error $\bm{\varepsilon}$ versus the iteration number.   We can see that the algorithm converges rapidly during the first five iterations and then
gives a steady approximation with the relative error $\bm{\varepsilon}=0.0505$.
The right plot in Fig \ref{fig-Bv-relative} shows the relative error $\bm{\varepsilon}$ for different  noise levels
 for  $20$ experiments. We can find the relative error becomes small  as the noise level decreases.

\begin{figure}
\centering
\subfigure[]{
\includegraphics[width=0.31\textwidth, height=0.36\textwidth]{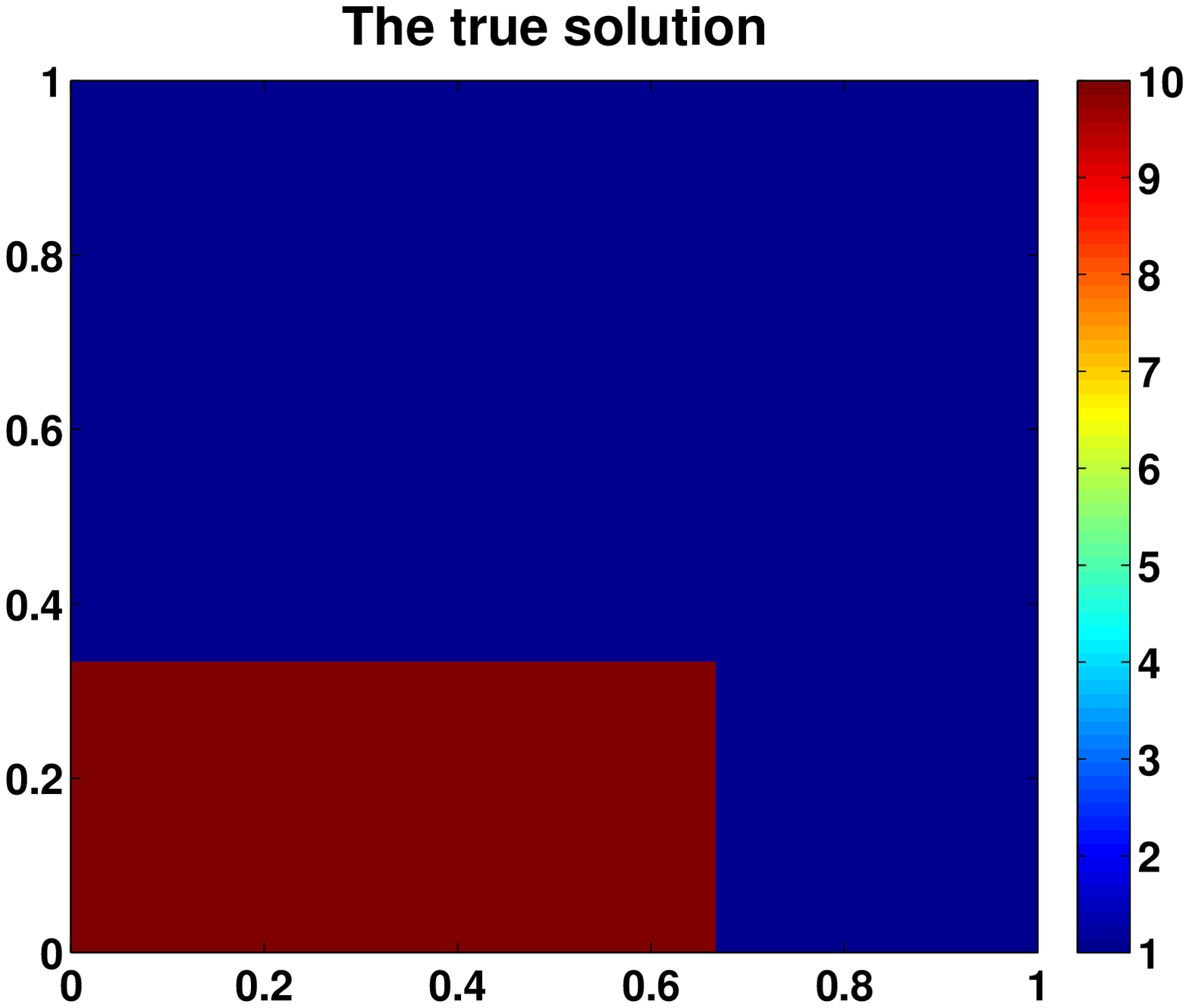}
}
\subfigure[]{
\includegraphics[width=0.31\textwidth, height=0.36\textwidth]{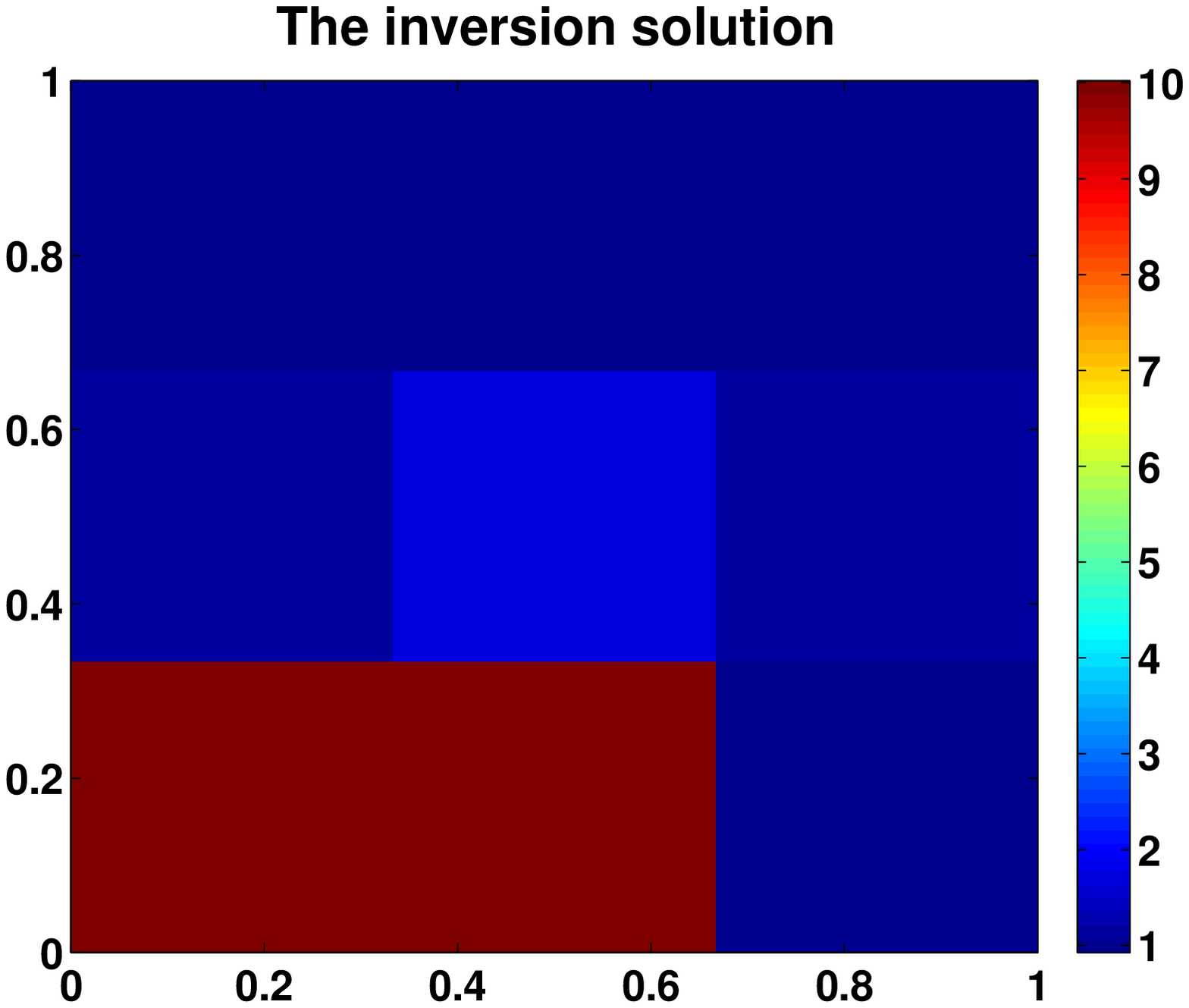}
}
\caption{The true solution and inversion solution in Section 6.2.}\label{fig-Bv-solution}
\end{figure}

\begin{figure}
\centering
\subfigure[]{
\includegraphics[width=0.31\textwidth, height=0.36\textwidth]{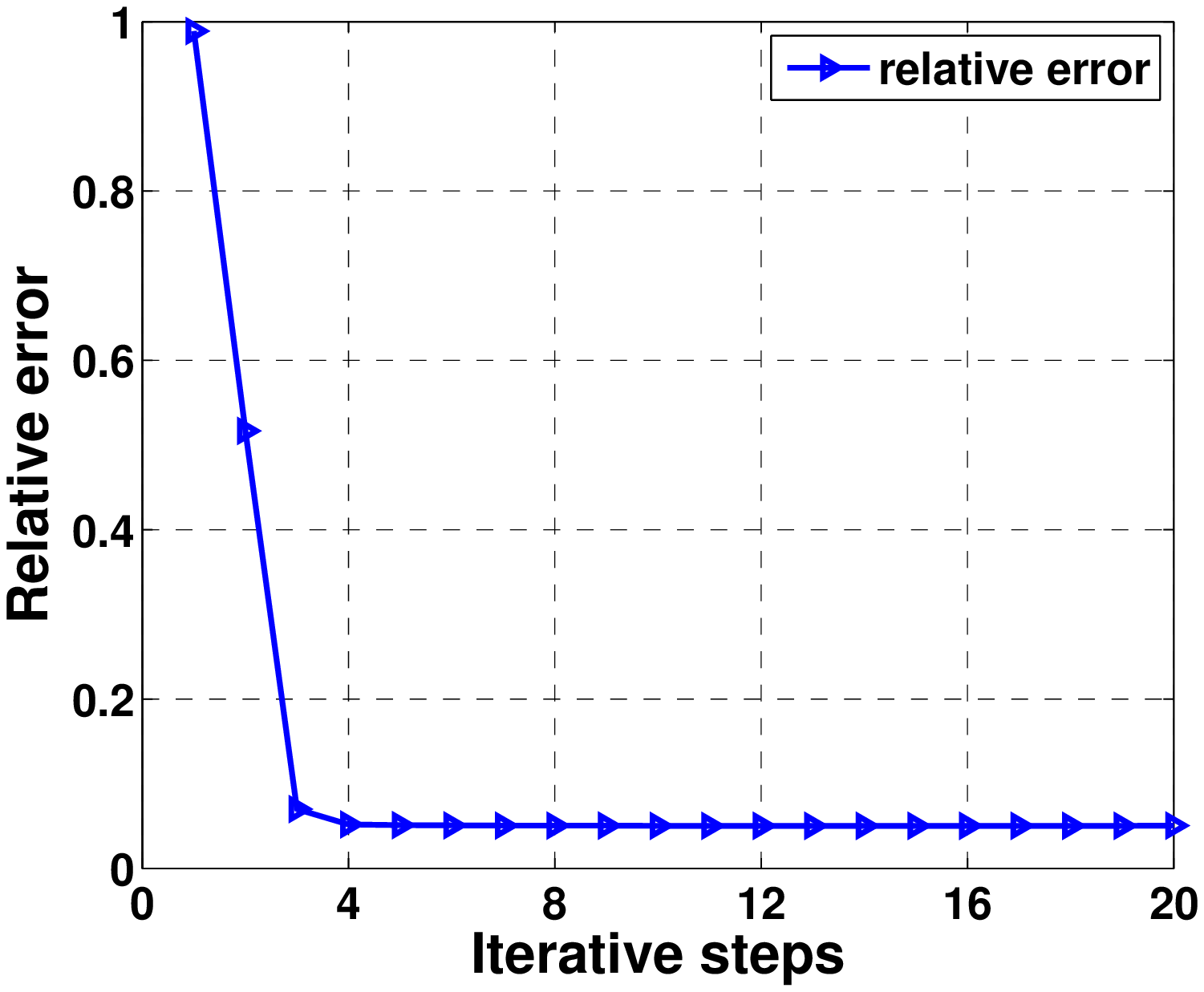}
}
\subfigure[]{
\includegraphics[width=0.31\textwidth, height=0.36\textwidth]{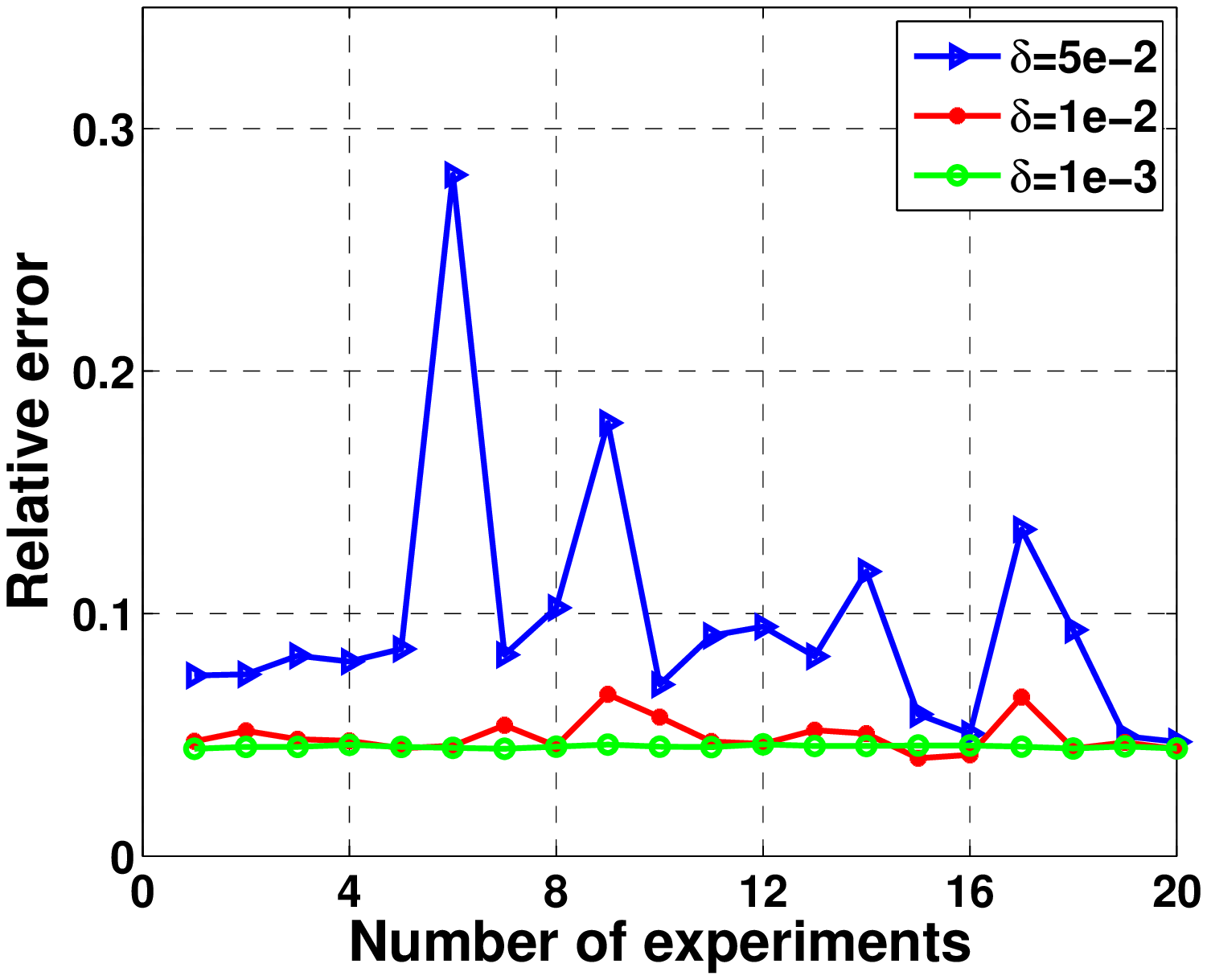}
}
\caption{(a): Relative errors versus iterative steps  in Section 6.2, (b): relative errors of
 20 experiments with different noise levels in Section 6.2.}\label{fig-Bv-relative}
\end{figure}

Table \ref{Bv-order} lists the result of the relative errors $\bm{\varepsilon}$ for different  fractional orders $\alpha$.
From the table, we find that the fractional order $\alpha$ has slight impact on the inversion solution, which is in agreement with the result in \cite{Sun-2017}.
The algorithm can be applied to the PDE model (\ref{model-fpde}) for different fractional orders between $0$ to $1$.
The numerical results for different regularization parameters are summarized in Table \ref{Bv-regular}. From this table, we observe that the regularization parameter $\gamma$ should be chosen in a proper range in order to get a better inversion solution. It shows that $\gamma=5e-3$ is the best parameter
among the four regularization parameter values for this example.

\begin{table}
\caption{Relative errors versus different fractional orders in Section 6.2.}\label{Bv-order}
  \centering
  \begin{tabular}{ccc}
  \hline
  $\alpha$&$a$&$\bm{\varepsilon}$\\\hline
  0.2 & (10.0795, 0.6058, 1.0384, 9.9579, 1.5174, 0.9600, 1.0074, 0.9828, 0.8238) & 0.0474 \\\hline
  0.4 & (10.0333, 1.2887, 0.9244, 9.9330, 1.6600, 1.0482, 1.0146, 1.1170, 0.9493) & 0.0515
  \\\hline
  0.6 & (10.0095, 0.7049, 1.0279, 9.9727, 1.5397, 0.9723, 0.9861, 1.0992, 1.0313) & 0.0435
 \\\hline
  0.8 & (9.9158, 1.5844, 0.9677, 9.9811, 1.7181, 1.0048, 1.1031, 1.1602, 0.9018) & 0.0664
  \\\hline
  \end{tabular}
\end{table}

\begin{table}
\caption{Relative errors versus different regularization parameters in Section 6.2.}\label{Bv-regular}
  \centering
  \begin{tabular}{ccc}
  \hline
  $\gamma$&$a$&$\bm{\varepsilon}$\\\hline
  5e-2 & (10.0686, 0.1806, 0.9822, 10.0180, 1.1028, 1.0226, 0.9728, 0.9968, 0.9995) & 0.0577\\\hline
  5e-3 & (9.9581, 0.9929, 1.0014, 10.0081, 1.5831, 0.8997, 0.9250, 1.1270, 1.0529) & 0.0426
 \\\hline
  5e-4 & (9.9714, 0.1806, 1.0161, 9.9695, 1.4353, 0.9684, 0.9898, 0.9885, 1.1179) & 0.0651
 \\\hline
  5e-5 & (10.0242, 2.0835, 1.0037, 9.9742, 0.7024, 0.9863, 1.0161, 0.1689, 0.9775) & 0.0972
 \\\hline
  \end{tabular}
\end{table}

\subsection{Inversion for piecewise smooth function}\label{3}
In this subsection, we consider the true reaction coefficient as  a piecewise smooth function, i.e.,
\[q(x,y)=\left\{\begin{array}{ll}
1,&\text{$0\leq x < 1/4,\
0\leq y \leq 1$},\\
12x-2,&\text{$1/4\leq x < 1/2,\
0\leq y \leq 1$},\\
4,&\text{$1/2\leq x < 3/4,\
0\leq y \leq 1$},\\
-12x+13, &\text{$3/4\leq x \leq1,\
0\leq y \leq 1$}.
\end{array}\right.\]

Different regularization methods may lead to different inversion solutions.
We first separately  employ the $L^{2}$ and $BV$ regularization methods to see if they can give a good reconstruction.
We use a $20\times 20$ grid to solve the forward problem and choose a piecewise constant basis $\{\psi_{\Omega_{i}}\}_{i=1}^{N_{1}}$ to represent $q$ as
\[
q(x,y)=\sum_{i=1}^{N_{1}}a_{i}\chi_{\Omega_{i}}(x).
\]
Such a basis have the property with
\[
\chi_{\Omega_{i}}(x)=
\left\{\begin{array}{ll}
1,\ \ x \in \Omega_{i},\\
0, \ \ else.
\end{array}\right.
\]
where $\Omega_{i}=[x_{i},x_{i+1}]$,\ $ i=0,\ 1,\cdots, N_{1}-1$.

Hence the dimension of $a$ is twenty.
{The measurement data are taken at time $t=(21:2:100)\cdot\Delta t$.}
The left plot in Fig \ref{fig-L2-Bv} shows the inversion results using classical $L^2$ penalty. By the plot we can see that the $L^{2}$ regularization solution is more smooth as $\beta$ becomes large.
The right plot in  Fig \ref{fig-L2-Bv}  shows the solutions using BV regularization method.
By the plot,  we find  that  the $BV$ regularization solution has more oscillations. This  is similar to stair-case effect. By the numerical test,  we find  that $L^2$ regularization method may recover smooth coefficient  and $BV$ penalty may produce accurate reconstruction of blocky images.
Thus, it may be helpful  to combine these two penalties together for obtaining a better inversion result for the case of  piecewise smooth functions.

Then, the mixed $L^{2}+BV$ regularization method is used here.
The true solution  and inversion solution are plotted in Fig \ref{fig-L2-Bv-piecewise} with $\beta=5.005e-3$ and
  $\gamma=1.005e-6$ fixed and we have the relative error $\bm{\varepsilon}=0.0578$.
The figure shows  that the multi-parameter model  gives a good inversion result. This is because the parameter $\beta$ and $\gamma$ can  make a tradeoff between $L^2$ and $BV$ penalty,
and share  the both effects from them.
In order to show the advantages of multi-parameter model further, we compare
the inversion solutions by the three regularization methods:   $L^{2}$ penalty,
 $BV$ penalty and $L^2+BV$ penalty.
The left plot in Fig \ref{fig-L2-Bv-conver} shows  the true solution  and the inversion solutions by the three different regularization methods. It can be obviously seen that the $L^{2}+BV$ regularization method can
most accurately  identify the discontinuous  structure of $q$ among the three methods. However, the $L^{2}$ regularization solution is over-smooth and the $BV$ regularization solution is highly oscillatory. Then we may conclude that the $L^{2}+BV$ regularization method is more appropriate for recovering piecewise smooth function.
The right plot  in Fig \ref{fig-L2-Bv-conver} illustrates that relative error $\bm{\varepsilon}$
versus the iteration number.  By the plot,  we see that the relative error
decreases dramatically during the first four iterations, which implies that the algorithm converges rapidly.
\begin{figure}
\centering
\subfigure[]{
\includegraphics[width=0.31\textwidth, height=0.36\textwidth]{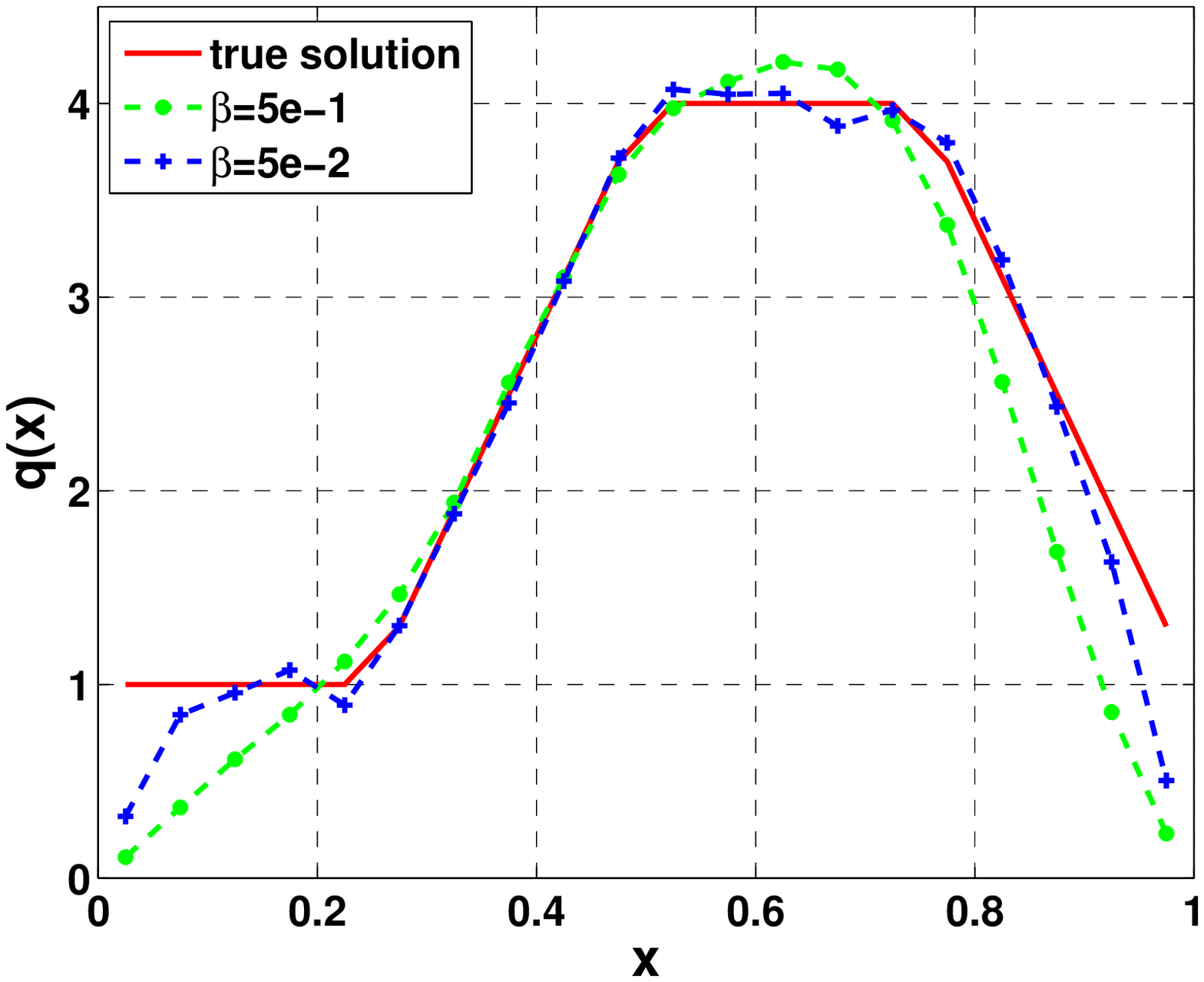}
}
\subfigure[]{
\includegraphics[width=0.31\textwidth, height=0.36\textwidth]{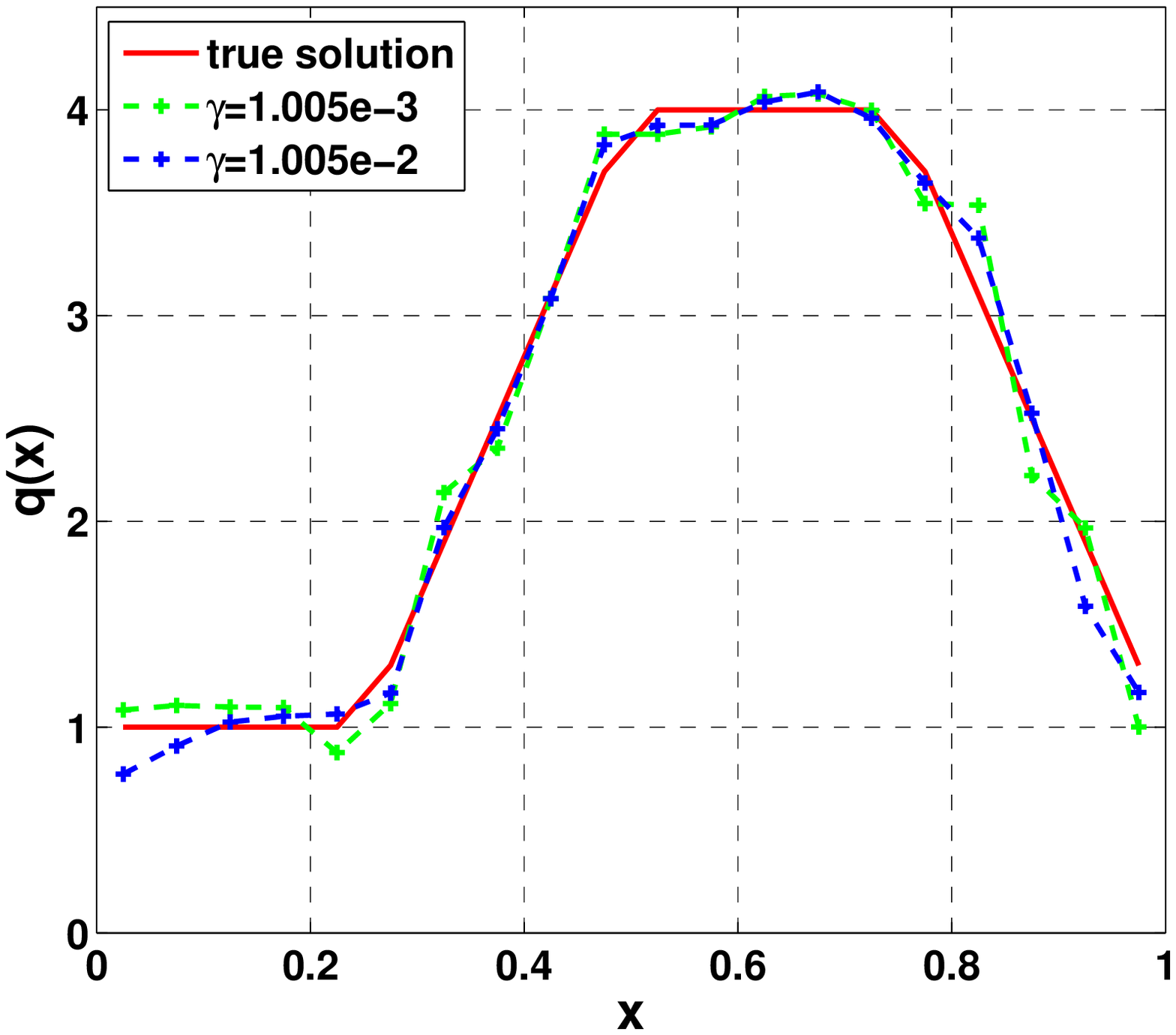}
}
\caption{(a): The true solution and inversion solution with different $\beta$ by $L^2$ regularization method, (b): the true solution and inversion solution with different $\gamma$ by $BV$ regularization method in Section 6.3.}
\label{fig-L2-Bv}
\end{figure}

\begin{figure}
\centering
\subfigure[]{
\includegraphics[width=0.31\textwidth, height=0.36\textwidth]{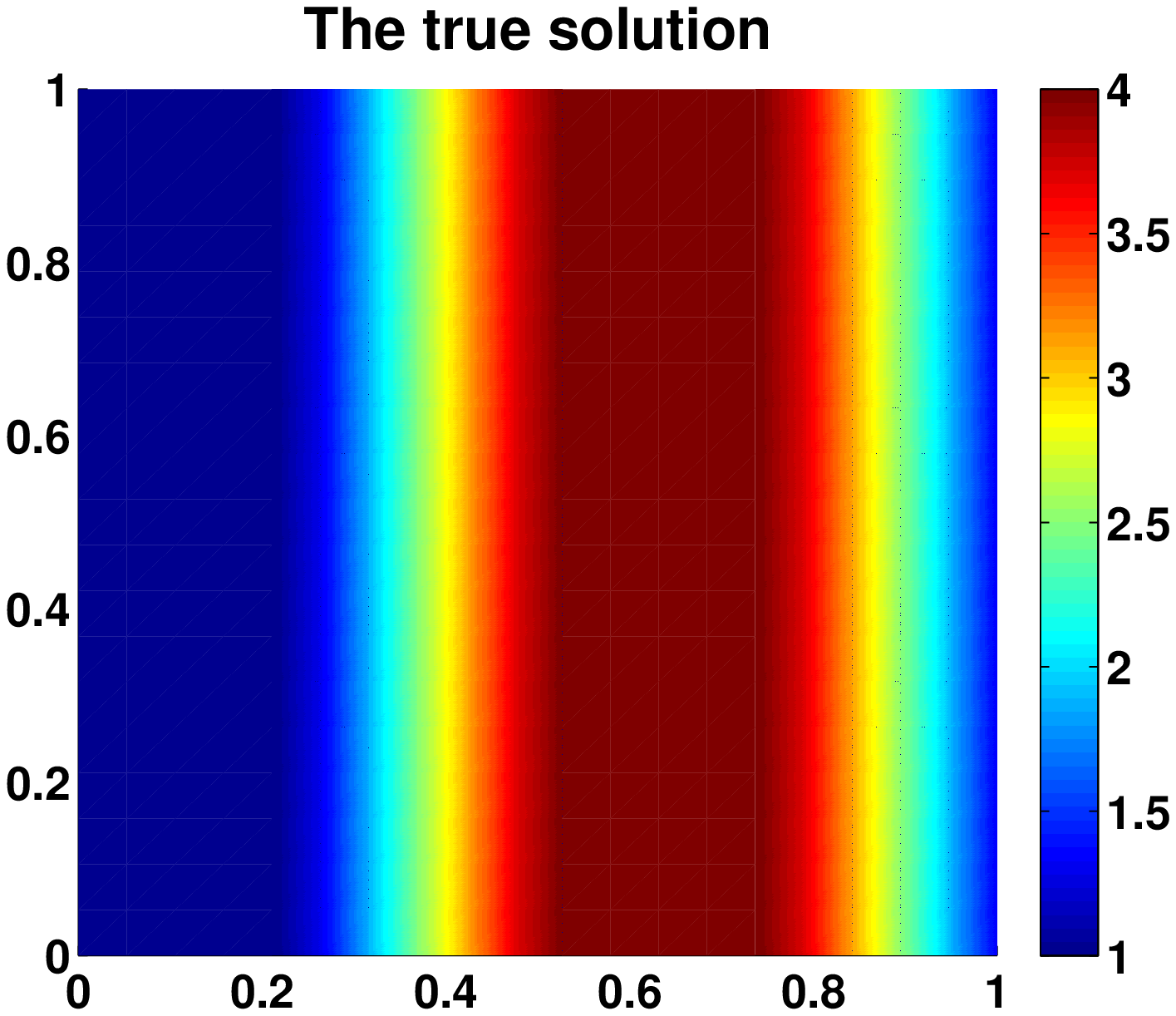}
}
\subfigure[]{
\includegraphics[width=0.31\textwidth, height=0.36\textwidth]{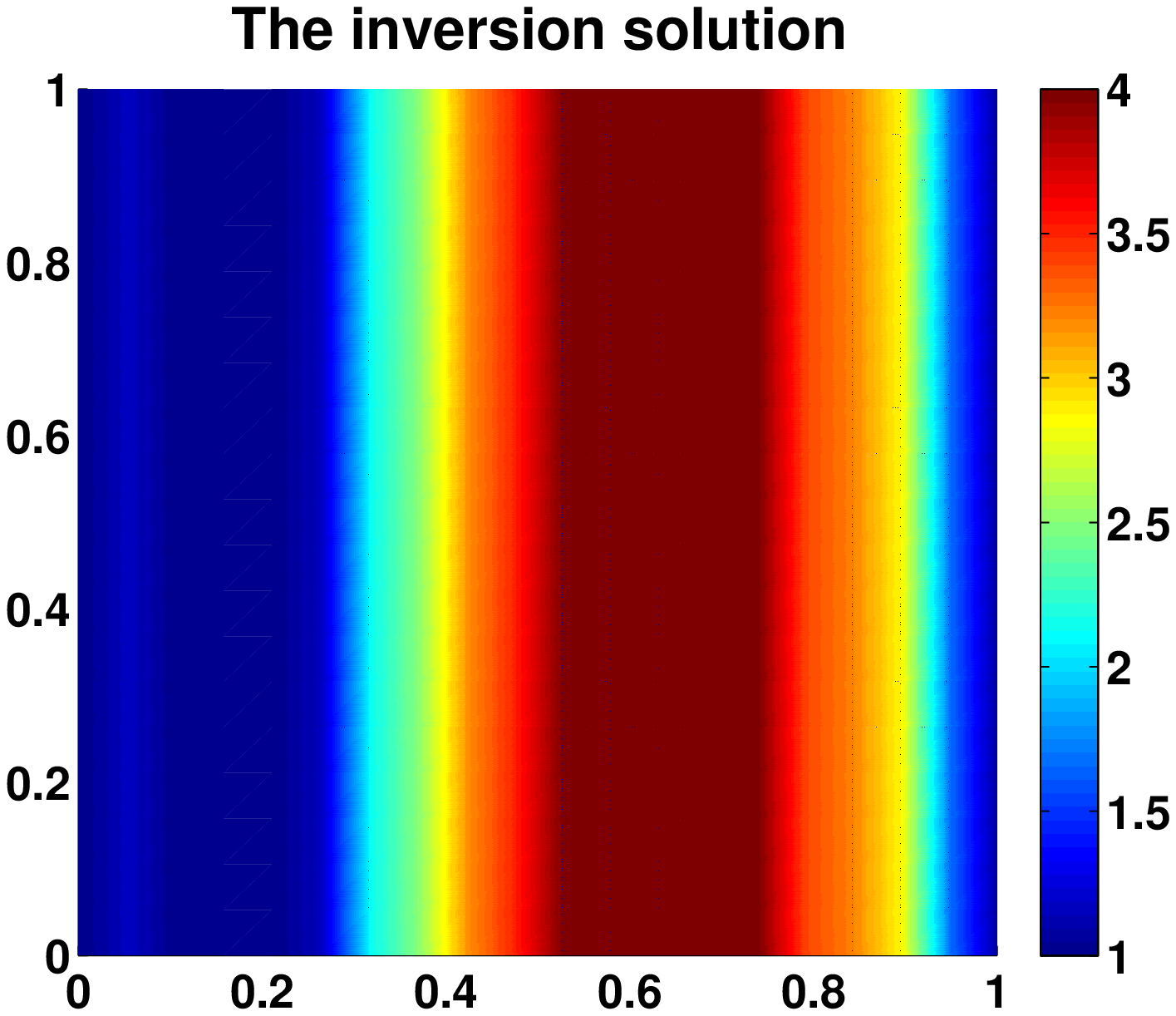}
}
\caption{The true solution and inversion solution in Section 6.3.}
\label{fig-L2-Bv-piecewise}
\end{figure}

\begin{figure}
\centering
\subfigure[]{
\includegraphics[width=0.31\textwidth, height=0.36\textwidth]{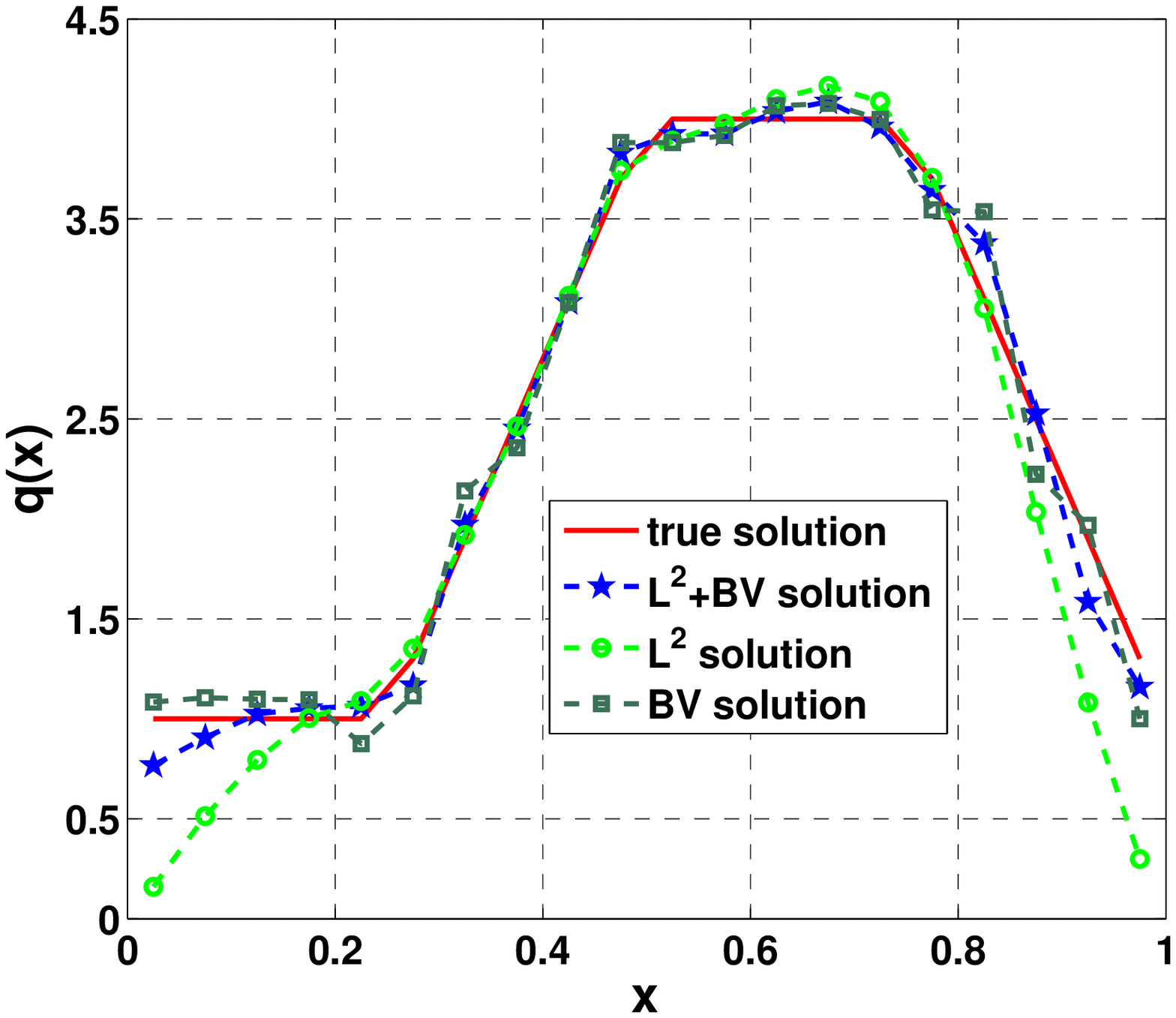}
}
\subfigure[]{
\includegraphics[width=0.31\textwidth, height=0.36\textwidth]{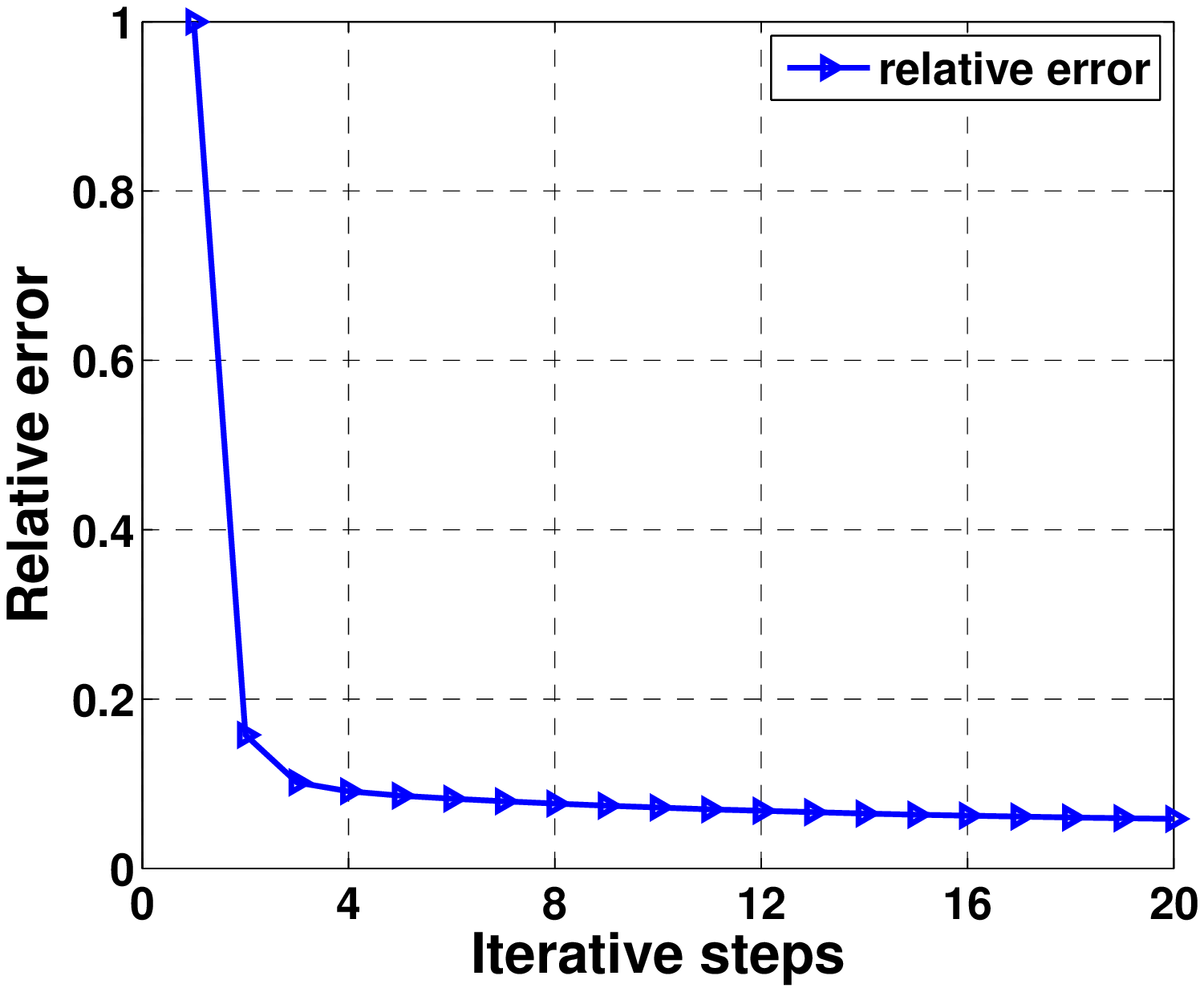}
}
\caption{(a): The true solution and inversion solution by $L^2+BV$, $L^{2}$, and $BV$ regularization methods, respectively. (b): relative errors versus iterative steps with $L^2+BV$ regularization method in Section 6.3.}
\label{fig-L2-Bv-conver}
\end{figure}

\section{Conclusions}
We considered an reaction coefficient inverse problem for time fractional diffusion equations.
We first proved that the forward operator is continuous with respect to the unknown parameter.
In practice, there exist  various types of the unknown reaction coefficients to recover.  We can
 solve a minimization problem through adding a muti-parameter penalty to the  fit-to-data functional.
To obtain the boundary flux  at each iteration,
the mixed finite element method was used for solving the forward problem, which can
give accurate flux.
By the extensive numerical simulations,
it could be concluded that different regularization methods are used  for recovering different types of the unknown coefficients.
The L-M algorithm would have difficulty in the  high dimension of unknown $q$. Thus it is
necessary to make  dimension reduction for parameters. To this end, we chose different basis functions for $q$ based on some priori information such that the unknown coefficient can be represented in a
low dimension space.
In the future,  we may consider the  regularization parameters depending on the given data,
and study  the coefficient and fractional order inverse problems for multi-term time-fractional diffusion equations.

\section*{Acknowledgments}
L. Jiang  acknowledges the support of Chinese NSF 11471107. G. Zheng acknowledges the support of Chinese NSF 11301168.

\end{document}